%% file: arxiv.tex
\documentclass[onefignum,onetabnum]{siamart171218}

\input{ex_shared}

\ifpdf
\hypersetup{
  pdftitle={Solving a sequence of eigenvalue problems using snapshot and row subsampling},
  pdfauthor={F. Hrobat, Y. Nakatsukasa}
}
\fi

\usepackage{amsmath}
\usepackage{amsfonts}
\usepackage{amssymb}
\usepackage{comment}
\usepackage{framed}
\usepackage{mathrsfs}
\usepackage{csquotes}
\usepackage{booktabs}
\usepackage{pgfplots}
\usepgfplotslibrary{groupplots}
\pgfplotsset{compat=1.18}
\usepackage{pgfplotstable}
\usepackage{float}
\usepackage{color}
\usepackage{tikz}
\usepackage{subcaption}
\tikzset{
    mark size=1pt
}
\usetikzlibrary{decorations.pathreplacing,calc}

\usepackage[normalem]{ulem}

\newcommand{\ignore}[1]{}
\newcommand{\C}{\mathbb{C}}
\newcommand{\R}{\mathbb{R}}

\renewcommand{\vec}[1]{\boldsymbol{#1}}

\makeatletter
\newcommand*\rel@kern[1]{\kern#1\dimexpr\macc@kerna}
\newcommand*\widebar[1]{%
  \begingroup
  \def\mathaccent##1##2{%
    \rel@kern{0.8}%
    \overline{\rel@kern{-0.8}\macc@nucleus\rel@kern{0.2}}%
    \rel@kern{-0.2}%
  }%
  \macc@depth\@ne
  \let\math@bgroup\@empty \let\math@egroup\macc@set@skewchar
  \mathsurround\z@ \frozen@everymath{\mathgroup\macc@group\relax}%
  \macc@set@skewchar\relax
  \let\mathaccentV\macc@nested@a
  \macc@nested@a\relax111{#1}%
  \endgroup
}
\makeatother

\definecolor{cbBlue}{HTML}{0072B2}
\definecolor{cbOrange}{HTML}{D55E00} 
\definecolor{cbGreen}{HTML}{009E73}
\definecolor{cbSkyBlue}{HTML}{56B4E9}
\definecolor{cbYellow}{HTML}{F0E442}
\definecolor{cbPurple}{HTML}{CC79A7}

\pgfplotsset{
    /pgfplots/cycle list={%
        {cbBlue, solid, thick},       
        {cbOrange, dashed, thick}, 
        {cbGreen, dotted, very thick},
        {black, dashdotted, thick},   
        {cbPurple, loosely dashed, thick},
        {cbSkyBlue, solid, thick}%
    },
    every axis legend/.append style={
        cells={anchor=west},
        draw=none,
    }
}

\begin{document}

\algnewcommand{\algorithmicgoto}{\textbf{go to}}%
\algnewcommand{\Goto}[1]{\algorithmicgoto~\ref{#1}}%

\maketitle

% REQUIRED
\begin{abstract}
 A sequence of eigenvalue problems requires solving $A_i\vec{x}_i= \lambda_i\vec{x}_i$ where $i \in \{1, \dots, N\}$ $N \gg 1$ and (usually) $\lambda_i$ is the largest or smallest eigenvalue of $A_i$. Often, the matrices $A_i$ come from a discretization of a smooth map 
 depending on a parameter $p \rightarrow A(p)$, that is $A_i = A(p_i)$, so the matrices $A_i$ change gradually along the sequence. Under this latter assumption, we introduce an efficient algorithm for solving sequences of eigenvalue problems. Our algorithm combines the ideas of reduced order model and row subsampling, as done in [E. Jones and Y. Nakatsukasa, arXiv, 2025] for linear systems. In addition, in our approach, the same set of rows is guaranteed to work well independently of the value of the parameter  $p$ or the matrix $A(p)$ as long as the snapshot matrix is unchanged and given that the row subset was chosen well at the beginning. Furthermore, we implement an update strategy of the snapshot matrix. Along these results, we present theoretical analysis of the snapshot matrix for eigenvalue problems, showing an exponential decay of the singular values, and we give  bounds on the error on eigenvectors for unseen value of the parameter. This analysis generalizes that in [M. Guido, D. Kressner, and P. Ricci, J. Sci. Comput., 2024] and shows that, regardless of the distribution of the snapshot points, the convergence is always (asymptotically) exponential in the number of snapshot points if the function $A(p)$ admits an analytical extension. We further generalize our algorithm for computing largest or smallest singular values and relative singular vectors and to generalized eigenvalue problems. We conclude with numerical results showing the efficiency of our approach.
\end{abstract}

% REQUIRED

% REQUIRED

\section{Introduction}

Given a matrix $A \in \C^{n \times n}$, finding the largest or smallest set of eigenvalues is a crucial problem in numerous applications, including engineering \cite[Chp. 10]{Saa2011}, PDEs \cite[\S 3]{Sch2001}, computational chemistry \cite{LeB2005}, network science \cite{Lan2012}, data analysis \cite{Has2009}. In particular, it is often required to solve a \textit{sequence of eigenvalue problems}.
These problems arise for example in kernel methods \cite[\S 3]{Sch2002}, spectral graph theory \cite{Chu1997}, manifold learning \cite{Bel2001}, deep learning \cite{Sag2016}, online learning \cite{Mit2013} and kernel methods for PDEs \cite{Ana2019}.
One of the most effective principles employed to tackle sequences of eigenvalue problems are model order reduction \cite{Mar2006, patera2007reduced, maday2002priori, hesthaven2015parametrized} as well as subspace recycling \cite{Soo2020,  Sir2016}, where the solution of the problem is sought in a subspace of small dimension. To limit the size of this subspace, Proper Orthogonal Decomposition (POD) may be employed \cite{Kun1999, Buc2013, Ber2023}. This idea has been already applied to parametric PDE eigenvalue problems \cite{Alg2025}, parametrized elliptic eigenvalue problems \cite{Fum2016, Hor2017} and data-driven reduced order modelling \cite{Alg2024, Ber2023/2}. More specifically, suppose we need to solve a sequence of eigenvalue problems of the form
\begin{equation}\label{eq:eig}
 A(p)\vec{x}(p) = \lambda(p)\vec{x}(p), 
\end{equation}
where $A(p) \in \C^{n\times n}$ for all values of $p$ in a compact domain, while $\lambda(p)$ is the largest/smallest eigenvalue. First, we can solve this problem for a selection of parameter values (the \textit{snapshot points)}
\begin{align*} A(p_0)\vec{x}(p_0) &= \lambda(p_0)\vec{x}(p_0), \\ A(p_1)\vec{x}(p_1) &= \lambda(p_1)\vec{x}(p_1), \\ &\vdots \\ A(p_{r-1})\vec{x}(p_{r-1}) &= \lambda(p_{r-1})\vec{x}(p_{r-1})
\end{align*}
where $r \ll n,N$.
Second, for all other values of $p$, we look for an eigenvector $\vec{x}(p)$ in the subspace
\[ \text{span}(\vec{x}(p_0), \vec{x}(p_1), \dots,\vec{x}(p_{r-1})).\]
Concretely, we have reduced the dimension of the space where $\vec{x}(p)$ might live from $n$ to $r$.

In order for this method to work, we theoretically need guarantees on how $A(p)$ changes (smoothly, analytically, continuously, \dots) and how the snapshots points are chosen. This problem has been addressed in various works regarding the Kolmogorov $n$-width \cite{Pin1985}. Under reasonable assumptions on analyticity of $A(p)$, Guido, Kressner and Ricci \cite{Gui2024} proved that this approach is effective in predicting future solution (i.e. for parameter value $p$ greater than all snapshot points) in the setting of linear systems. This analysis is particularly useful for time-dependent eigenvalue problems, where the parameter $p$ has a natural ordering, as in the case of time.

Define the \textit{snapshot matrix} as
\[ X = [\vec{x}(p_0), \dots , \vec{x}(p_{r-1})] \in\mathbb{C}^{n\times r}.\]
In the setting of eigenvalue problem, to compute the eigenpair $(\lambda(p), \vec{x}(p))$ we 
consider the variational formulation~\cite[Thm. 11.4.2]{Par1998}\cite[Eq. 1.10]{Nak2024} of the classical Rayleigh-Ritz process:
first solve the least-squares problem
\[ M(p) = \operatorname{argmin}_{M \in \C^{r \times r}} \| A(p)X-XM\|_F \]
and return $(\Tilde{\lambda}(p),\Tilde{\vec{x}}(p))$ where $\Tilde{\vec{x}}(p) = X\vec{y}(p)$ and $(\Tilde{\lambda}(p), \vec{y}(p))$ is an eigenpair of $M(p)$, i.e.
\[ M(p)\vec{y}(p) = \Tilde{\lambda}(p)\vec{y}(p)\]
and $\Tilde{\lambda}(p)$ is the largest/smallest eigenvalue of $M(p)$. A popular strategy when the number of snapshot points is large is Proper Orthogonal Decomposition (POD): first we select the first left singular vectors $Q$ of $X$ and we use $Q$ instead of $X$ to project $A(p)$ \footnote{Note that the solution $M(p)$ for the least-squares problem $M(p) = \operatorname{argmin}_{M \in \C^{r \times r}} \| A(p)X-XM\|_F $  depends on $X$, but the returned eigenpairs only depend on $\operatorname{span}(X)$. In fact, if $\overline{X}$ is such that $\overline{X}=XR$ for some invertible $R \in \C^{r \times r}$, then the solution of  $\operatorname{argmin}_{M \in \C^{r \times r}.} \| A(p)\overline{X}-\overline{X}M\|_F \|$ is equal to $R^{-1}M(p)R$, which has the same eigenvalues of $M(p)$ and the corresponding eigenvector $\overline{\vec{y}}(p)$ is equal to $R^{-1}\vec{y}(p)$, which implies $X\vec{y}(p) = \overline{X}\overline{\vec{y}}(p)$. In particular, we can first orthogonalize $X$ and then solve the least-squares problem without changing the computed eigenpairs.}, possibly truncating $Q$ depending on the value of the singular values.
This is closely related to the classic Rayleigh-Ritz method for eigenvalue problem \cite[\S 4.3]{Saa2011}.

In this work we combine the snapshot
method and the method of sketched Rayleigh-Ritz \cite[\S 1.3.1]{Nak2024}.
In particular, we will consider a subsampling matrix $S \in \R^{s \times n}$ and replace $M(p)$ with
\[ \widehat{M}(p) = \operatorname{argmin}_{M \in \C^{r \times r}} \|S(A(p)X-XM)\|_F.\]
This method has the appealing advantage that, once we have computed $S$ (i.e we have chosen the row indices), we can compute $SA(p)$ without looking  at the entire matrix $A(p)$, besides the fact that the least-square problem has now dimension $s \times r$ rather than $n \times r$.
The remainder of this paper is organized as follows: in \Cref{sec: snapmatrix}, we provide theoretical results on the snapshot matrix and we prove that the singular values decay exponentially if $A(p)$ is analytic, and the rate of decay depends on how large the domain of analyticity of $A(p)$ is. In \Cref{sec: theory} we provide an extensive theoretical analysis of the subspace generated by the snapshot points. In particular, we show how the distance between an eigenvector associated with a new parameter value to this subspace decays when the snapshot points are the equispaced points or Chebyshev points in $[-1,1]$. Furthermore, we show that this rate is asymptotically exponential for any distribution of (distinct) snapshot points. \Cref{sec: algo} describes our algorithm  in detail, and in \Cref{sec: imple} we provide some comments on the implementation. In \Cref{sec: related} we list some related algorithms in the literature. In \Cref{sec: svd} we generalize our algorithm to the SVD and generalized eigenvalue problems. To conclude, in \Cref{sec: experiments}, we show numerical experiments.

\subsubsection*{Terminology}
In the following, we will use
\begin{itemize}
    \item \textit{time-dependent eigenvalue problems} to indicate that the matrices $A(p)$ can only be accessed sequentially, for example if $p$ represents time;
    \item \textit{parameter-dependent eigenvalue problems} to indicate that the matrices $A(p)$ can be accessed independently of the value of $p$, without restrictions;
    \item \textit{sequence of eigenvalue problems} to indicate both time-dependent and \\ parameter-dependent eigenvalue problems.
\end{itemize}

\section{Snapshot matrix for eigenvalue problems}\label{sec: snapmatrix}

In this section we outline the assumption we make to conduct the theoretical analysis of our algorithm and we show the resulting theoretical properties of the snapshot matrix $X$. We work under similar assumptions as in \cite{Gui2024}.
We assume that the parameter domain is $[-1,1]$ for simplicity.

\begin{assumption}[{from \cite[Assumption $1$]{Gui2024}}]\label{assumption}
    Consider the open Bernstein ellipse $E_{\rho} \subset \C$ for $\rho >1$, that is, the open ellipse with foci $\pm 1$ and semi-minor/-major axes summing up to $\rho$. We assume that $A: [-1,1] \rightarrow \C^{n \times n}$ admits an analytic extension on $E_{\rho}$ which is also continuous on the closure of $E_{\rho}$, denoted by $\widebar{E}_{\rho}$. 
    Furthermore, we assume that the eigenvalue $\lambda(p)$ we are considering is simple and that 
    \[ \text{gap} := \inf_{p \in E_{\rho}} \min_{\lambda \in \Lambda(A(p))\text{\textbackslash}{\lambda(p)}} |\lambda(p) - \lambda| > 0.\]
\end{assumption}

With a slight abuse of notation, we will indicate by $A$ the analytic extension of $A$ defined on $E_{\rho}$. Under this assumption we have the following result.

\begin{lemma}
    Under \Cref{assumption}, there exists $\Delta$ such that
    \[ \|A(p)\|_2 \leq \Delta \]
    for all $p \in \widebar{E}_{\rho}$.
    Furthermore, there exists a continuous map $\vec{x}: \widebar{E}_{\rho} \rightarrow \C^n$  such that $\vec{x}|_{E_{\rho}}$ is analytic for all $p \in \widebar{E}_{\rho}$, and 
    \[ A(p)\vec{x}(p) = \lambda(p)\vec{x}(p).\]
\end{lemma}
\begin{proof}
    The first claim follows as $A$ is continuous and $\widebar{E}_{\rho}$ is compact. The second claim follows from \cite[\S 2]{Kat1976}.
\end{proof}
Define 
\begin{equation}\label{eq:kappa}
\kappa_{\rho} = \max_{t \in \partial E_{\rho}}\|\vec{x}(t)\|.    
\end{equation}
Thanks to the previous lemma, we can choose the snapshot matrix

\[ X = [\vec{x}(p_0), \dots , \vec{x}(p_{r-1})] \in \C^{n \times r}\]
where $\vec{x}(p)$ depends analytically on $p$.
As a consequence, we can prove that the singular values of $X$ decay exponentially.

\begin{theorem}\label{thm:singdecay}
    Assume \Cref{assumption}. Let $X \in \C^{n \times r}$ be the snapshot matrix given by
    \[ X = [\vec{x}(p_0), \vec{x}(p_1), \dots , \vec{x}(p_{r-1})]\]
    and the map $\vec{x}$ is analytic. Then the $k$-th singular value  $\sigma_k(X)$ of $X$ satisfies

    \[ \sigma_k(X) \leq \kappa_{\rho}\frac{2\rho\sqrt{r}}{1 - \rho^{-1}} \rho^{-k}\]
    where $\kappa_\rho$ is as in~\cref{eq:kappa}.
    
\end{theorem}
\begin{proof}
    This follows directly from Theorem 2.4 in \cite{Kre2011} applied to $X$ (there named $B$).
\end{proof}

% =========================================================
% PLOT 3
% =========================================================
\begin{figure}[ht]
    \centering
    \begin{tikzpicture}
        \begin{semilogyaxis}[
            width=0.6\textwidth,
            title={Decay of the smallest singular value of the snapshot matrix},
            xlabel={Number of snapshot ($r$)},
            ylabel={$\sigma_{\min}$},
            grid=major,
            legend pos=south west,
            legend style={font=\footnotesize, cells={anchor=west}}
        ]
        
        \addplot[color=blue, mark=square, thick, only marks] 
            table[x expr=\coordindex+4, y=sing_cheb] {theory/1000_3.1_0.8765.tex};
        \addlegendentry{$\sigma_{\min}\ X_{\text{cheb}}$}
        
        \addplot[color=red, mark=*, thick] 
            table[x expr=\coordindex+4, y=sing_equi] {theory/1000_3.1_0.8765.tex};
        \addlegendentry{$\sigma_{\min}\ X_{\text{equi}}$}
        
        \addplot[color=black, dashed, thick, mark=none] 
            table[x expr=\coordindex+4, y=rate_sing_cheb] {theory/1000_3.1_0.8765.tex};
        \addlegendentry{estimated rate}
        
        \end{semilogyaxis}
    \end{tikzpicture}
    \caption{Plot of smallest non-zero singular value of the snapshot matrix as a function of the number of snapshot points. The targeted eigenvalue is the rightmost one. The matrix $A(p)$ is equal to $A(p) = M_1 +p\frac{M_2}{4}$, where we obtained $M_1, M_2$ by taking the symmetric part of two random matrices of size $1000\times 1000$. The parameter domain is $[-1,1]$. $X_{cheb}$ is the snapshot matrix where the snapshot points are the Chebyshev points while $X_{equi}$ is the snapshot matrix where the snapshot points are equispaced points. The estimated rate of decay is $\rho \approx 3.1$, which is observed to reflect the actual behavior}.\label{fig: singdecay}
\end{figure}
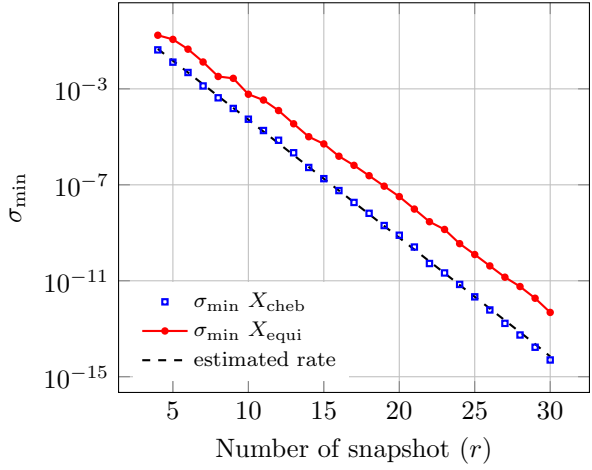

\Cref{thm:singdecay} justifies the POD approach,
which consists in only taking the first left singular vectors of the snapshot matrix when $r$ is large.
Note that assuming a gap in \Cref{assumption} is fundamental for the analyticity of $\vec{x}$. If the eigenvalue $\lambda(p)$ is not separated from the rest of the spectrum, we can instead consider a larger set of $m$ eigenvalues of $A(p)$
\[ \Lambda(p) = \{\lambda_1(p), \dots, \lambda_m(p)\}.\]
The snapshot matrix is then chosen as
\[ X = [\vec{x}(p_0), \vec{x}(p_1), \dots , \vec{x}(p_{r-1})] \in \C^{n \times mr}\]
where $\vec{x}(p_i) \in \C^{n \times m}$ and $\text{span}(\vec{x}(p_i))$ is the eigenspace associated with $\Lambda(p)$. 

In this case, we make the following assumption.

\begin{assumption}[{\Cref{assumption} for multiple eigenvalues}]\label{assumptionblock}
    Consider the open Bernstein ellipse $E_{\rho} \subset \C$ for $\rho >1$, that is, the open ellipse with foci $\pm 1$ and semi-minor/-major axes summing up to $\rho$. We assume that $A: [-1,1] \rightarrow \C^{n \times n}$ admits an analytic extension on $E_{\rho}$ and continuous on the closure of $E_{\rho}$, denoted as $\widebar{E}_{\rho}$. 
    Furthermore, we assume that the eigenvalues $\{ \lambda_1(p), \dots , \lambda_m(p) \}$ we are considering are separated from the rest of the spectra of $A(p)$ in this way: there exists a closed curve $\Gamma$ independent of $p$ such that $\Gamma$ encloses $\Lambda(p)$ for all $p$ while the rest of the spectrum of $A(p)$ lies outside $\Gamma$. 
\end{assumption}

The following theorem guarantees that we can choose $\vec{x}(p_i)$ to be analytic if $A(p)$ is analytic.

\begin{theorem}\label{thm:analyticblock}
    Under \Cref{assumptionblock}, we can choose $\vec{x}(p) \in \C^{n \times m}$ to depend analytically on $p$.
\end{theorem}

\begin{proof}
    According to \cite[Chp. 2 \S 1.4]{Kat1976}, the eigenspace
    \[ M(p) = \text{span}\left(\frac{1}{2\pi i}\int_{\Gamma}(zI - A(p))^{-1}dz\right)\]
    is an analytic vector bundle, where $\Gamma$ is defined in \Cref{assumptionblock}. As $E_{\rho}$ is open and simply connected, the vector bundle $M(p)$ is holomorphically trivial, hence, thanks to \cite[Thm. 30.1]{For1981}, there exist $m$ 
    global holomorphic sections $[\vec{x}_1(p), \dots , \vec{x}_m(p)] \in M(p)$ such that for
   each point $p \in E_{\rho}$ the vectors $[\vec{x}_1(p), \dots , \vec{x}_m(p)]$ are linearly independent.
\end{proof}

\Cref{thm:analyticblock} guarantees that we can select a basis for the eigenspace we are considering that depends analytically on $p$ in $E_{\rho}$. \Cref{thm:singdecay} can be generalized also to this setting ($mr$ takes the place of $r$).

Furthermore, if $A(p)$ is real and symmetric for $p \in E_{\rho} \cap \R$, then we can assume that the eigenvectors $[\vec{x}_1(p), \dots, \vec{x}_m(p)]$ are also orthonormal on  $E_{\rho} \cap \R$ (under the assumption that $\vec{x}_i(p)^T\vec{x}_i(p) \notin \R^{-}$ in $E_{\rho}$ for all $i = 1, \dots ,m$). This is because we can run Gram-Schmidt on the basis $[\vec{x}_1(p), \dots, \vec{x}_m(p)]$ on $E_{\rho} \cap \R$ and preserve the analyticity as $1/\sqrt{(\vec{x}_i(p)^T\vec{x}_i(p))}$ (without complex conjugate) is an analytic map as long as $\vec{x}_i(p)^T\vec{x}_i(p) \notin \R^{-}$ in $E_{\rho}$.

\section{Theoretical results} \label{sec: theory}
\subsection{Theoretical results on interpolation and extrapolation}

In this section, we generalize the theoretical analysis conducted in \cite{Gui2024} to the setting of eigenvalue problems. In particular, we first want to prove the subspace spanned by the snapshot matrix almost contains the desired eigenvector if we choose the snapshot points 
at equispaced points:
\begin{equation}
    \label{eq:equispaced}
p_{0} = -1, \ p_1 = -1 + \Delta t, \ \dots ,\ p_{r-2} = 1- \Delta t, \ p_{r-1} = 1    
\end{equation}
where
\[ \Delta t = \frac{2}{r-1}.\]

We will conduct the theoretical analysis for a simple isolated eigenvalue, but it can be extended to multiple eigenvalues under \Cref{assumptionblock}.
Below, we focus on bounding $\|(I-XX^{\dagger})\vec{x}(p_*)\|_2$; this is a natural quantity that measures how much the span of $X$ contains $\vec{x}(p_*)$ since $XX^{\dagger}\vec{x}(p_*)$ is the orthogonal projection of $\vec{x}(p_*)$ onto $\mbox{span}(X)$. 

\begin{theorem}\label{thm: convergence}
Define the snapshots points as in~\cref{eq:equispaced}, and let $p_*$ be the new parameter value for which we want to solve the eigenvalue problem~\cref{eq:eig}. Under \Cref{assumption}, it holds that, if $p_* \in [-1,1]$ then
\begin{equation}
\bigl\| (I-XX^{\dagger})\vec{x}(p_*) \bigr\|_2
\le
2\
\frac{\kappa_{\rho}}{1 - \rho^{-1}}\left(
1 +
C(r,R)\sqrt{(R+1)}
\right)\rho^{-(R+1)}
\end{equation}
else if $p_* \in (1, \frac{\rho + \rho^{-1}}{2})$ then 
\begin{equation}
\bigl\| (I-XX^{\dagger})\vec{x}(p_*) \bigr\|_2
\le
2 \kappa_{\rho}
\left(
\frac{1}{1 - \delta}
+
\frac{C(r,R)}{(1 - \rho^{-1})\sqrt{\rho^2\delta^2-1}}
\right)
\, \delta^{R+1},
\end{equation}
where $
C(r,R) = \frac{5\sqrt{5}\sqrt{2R+1}\sqrt{r}}{\sqrt{2(r-1)}}
$, for any $R \le \frac{1}{2}\sqrt{r-1},\ \delta = \frac{p_* + \sqrt{p_*^2 - 1}}{\rho} < 1.$
\end{theorem}

\begin{proof}
    We follow the argument in \cite[Thm $5$]{Gui2024}. Let $\vec{p}(t)$ be a vector-valued polynomial. Assume also that $\vec{p}(t) \in \mbox{span}(X)$ for all $t$. Since the orthogonal projection $XX^{\dagger}\vec{x}(p_*)$ finds the closest point in $\mbox{span}(X)$ to $\vec{x}(p_*)$, we have
    \[ \bigl\| (I-XX^{\dagger})\vec{x}(p_*) \bigr\|_2
\le \|\vec{x}(p_*)-\vec{p}(p_*)\|_2.\]
In particular, an upper bound on $\bigl\| (I-XX^{\dagger})\vec{x}(p_*) \bigr\|_2$ is given by the best polynomial approximation in $p_*$. The result in \cite[Thm $5$]{Gui2024} gives un upper bound on this quantity.

Let $\vec{x}_R(t) = \sum_{k=0}^R \vec{c}_k T_k(t)$ be the truncated Chebyshev series 
(which converges geometrically by the analyticity assumption~\cite[Thm. 8.2]{Tre2020})
of $\vec{x}(t)$ up to degree $R$, and let $\vec{p}_R(t) = \sum_{k=0}^R \vec{a}_k T_k(t) \in \text{span}(X)$ be the least-square polynomial constructed on the snapshot points.
According to \cite[Lemma 4]{Gui2024} the coefficients satisfy $\|\vec{c}_k\|_2 \le 2\kappa_\rho \rho^{-k}$. Furthermore, defining $C_x = [\vec{c}_0, \dots, \vec{c}_R]$, $C_p = [\vec{a}_0, \dots, \vec{a}_R]$, and $\vec{q}_R(t) = [q_0(t), \dots, q_R(t)]^T$, we have that 
\begin{equation}\label{eq: cxcp} \|C_x - C_p\|_2 \le \frac{2\kappa_\rho}{1 - \rho^{-1}} C(r,R) \rho^{-(R+1)}.
\end{equation}
 By the triangle inequality, the error can be split as
\[
\|\vec{x}(p_*) - \vec{p}_R(p_*)\|_2 \le \|\vec{x}(p_*) - \vec{x}_R(p_*)\|_2 + \|\vec{x}_R(p_*) - \vec{p}_R(p_*)\|_2.
\]
We divide the proof into two cases.

\noindent\textbf{Case 1: $\boxed{p_* \in [-1,1]}$.}
Since $p_* \in [-1,1]$, the Chebyshev polynomials satisfy $|T_k(p_*)| \le 1$ for all $k \ge 0$. It holds that
\[
\|\vec{x}(p_*) - \vec{x}_R(p_*)\|_2 \le \sum_{k=R+1}^\infty \|\vec{c}_k\|_2 |T_k(p_*)| \le 2\kappa_\rho \sum_{k=R+1}^\infty \rho^{-k} = \frac{2\kappa_\rho}{1 - \rho^{-1}} \rho^{-(R+1)}.
\]
For the second term,  we have:
\[
\|\vec{x}_R(p_*) - \vec{p}_R(p_*)\|_2 = \|(C_x - C_p)\vec{q}_R(p_*)\|_2 \le \|C_x - C_p\|_2 \|\vec{q}_R(p_*)\|_2.
\]
and we can substitute \eqref{eq: cxcp} and $\|\vec{q}_R(p_*)\|_2 = \sqrt{\sum_{k=0}^R |T_k(p_*)|^2} \le \sqrt{R+1}.$

\noindent\textbf{Case 2: $\boxed{p_* \in \left(1, \frac{\rho + \rho^{-1}}{2}\right)}$.} Outside the interval $[-1,1]$, the growth of the Chebyshev polynomials is bounded by $|T_k(p_*)| \le (\rho \delta)^k$, where $\delta = \frac{p_* + \sqrt{p_*^2 - 1}}{\rho} < 1$~\cite{mason2010chebyshev}. Consequently:
\[
\|\vec{q}_R(p_*)\|_2 = \sqrt{\sum_{k=0}^R |T_k(p_*)|^2} \le \sqrt{\sum_{k=0}^R (\rho\delta)^{2k}} = \sqrt{\frac{(\rho\delta)^{2(R+1)} - 1}{\rho^2\delta^2 - 1}} < \frac{(\rho\delta)^{R+1}}{\sqrt{\rho^2\delta^2 - 1}}.
\]

By applying this bound to the same error splitting we obtain the second inequality of the theorem.
\end{proof}

Note that the two results in \Cref{thm: convergence} account for parameter-dependent and time-dependent eigenvalue problems respectively, as the latter allows us to extrapolate in $p$, while the former focuses on interpolating.
A result analogous to \Cref{thm: convergence} can be proved for snapshot points that are not equispaced but are equal to the Chebyshev points. Although this is not the standard scenario for time-dependent problems (as the snapshot points are not nested as $r$ increases), it is useful for parameter-dependent problems, and a natural choice given the prominence of Chebyshev interpolation for approximating univariate functions, where it is vastly superior to equispaced interpolation~\cite{Tre2020}; however, in our context, there is much less evidence (theoretically and empirically) that there is a substantial difference. In particular, the Runge phenomenon is not observed as $\|(I-XX^{\dagger})\vec{x}(p)\|$ is tipically decreasing in the number of parameters.

\begin{theorem}[\Cref{thm: convergence} for Chebyshev points]\label{thm: convergencecheb}
Define the \textit{snapshots points} as the Chebyshev points, i.e. 
\[ p_i = \cos\left(\frac{i \pi}{r-1}\right), \quad i = 0, \dots, r-1\]
and let $p_*$ be the new parameter value for which we want to solve the eigenvalue problem. Under \Cref{assumption}, it holds that, if $p_* \in [-1,1]$ then

\begin{equation}
\bigl\| (I-XX^{\dagger})\vec{x}(p_*) \bigr\|_2
\le
2\
\frac{\kappa_{\rho}}{1 - \rho^{-1}}\left(
1 +
\frac{\sqrt{2}r}{\sqrt{r-1}}
\right)\rho^{-r}
\end{equation}
else if $p_* \in (1, \frac{\rho + \rho^{-1}}{2})$ then 
\begin{equation}
\|(I-XX^\dagger)\vec x(p_*)\|_2
\le
2\kappa_\rho
\left(
\frac{1}{1-\delta}
+
\frac{\sqrt{r}}
{(1-\rho^{-1})\sqrt{(\rho^2\delta^2-1)(r-1)}}
\right)\delta^{r}
\end{equation}
for any $\delta = \frac{p_* + \sqrt{p_*^2 - 1}}{\rho} < 1.$    
\end{theorem}

\begin{proof}
    The proof is the same as for \Cref{thm: convergence} and \cite[Thm 3]{Gui2024} but, for bounding $\|C_x - C_p\|_2$ we use that 
    \[ \sigma_{\min}(Q_R(\vec{p}_{cheb})) \geq \sqrt{\frac{r-1}{2}}\]
    as stated in \cite{Str1999}, also reported in \cite{Dem2019}. The $Q_R(\vec{p}_{cheb})$ is the matrix whose entry $(i,j)$ is equal to the $j$-th Chebyshev polynomial evaluated at the $i$-th Chebyshev point. 
\end{proof}

This is in accordance with the scalar polynomial approximation \cite[Thm 8.2]{Tre2020}. As for our knowledge, this is a novel approach to show the convergence properties of a reduced basis method. We can further generalize \Cref{thm: convergencecheb} by proving that the error decays exponentially in $r$ for any distribution of parameters that converge to a unit positive Borel measure $\mu$ in $[-1,1]$ in the weak-star topology.
The \textit{logarithmic potential} associated with a measure $\mu$ on $[-1,1]$ is defined as the function~\cite{Saf2010}
\[ U^{\mu}(z) = \int_{-1}^1\log\frac{1}{|z-t|}d\mu(t). \]

\begin{theorem}\label{thm: allmeasure}
    Consider a unit positive Borel measure $\mu$ on $[-1,1]$. Define the snapshot points $p_0, \dots, p_{r-1}$ such that $p_i \neq p_j$ and the empirical measure
    \[ \mu_r = \frac1r\sum_{i = 0}^{r-1} \delta(p_i)\]
    converges to $\mu$ in the weak-star topology as $r\rightarrow \infty$. Let $p_*$ be a new parameter value (inside or outside $[-1,1]$) such that $U^\mu(p_*)<+\infty$.
    Then,
    \[ \limsup_{r \rightarrow \infty}\bigl\| (I-XX^{\dagger})\vec{x}(p_*) \bigr\|_2^{1/r} \leq e^{-c} <1 \]
    for some constant $c>0$ that may depend on $p_*$,  under the assumption that $\rho$ in \Cref{assumption} is large enough such that 
    \[ \max_{z \in \partial E_{\rho}}U^{\mu}(z)  <U^{\mu}(p_*).\]
\end{theorem}
\begin{proof}
The proof combines the Hermite interpolation formula \cite[Eq. 4.2]{Saf2010} \cite[\S 11]{Tre2020}. with logarithmic potential theory \cite{Saf2010}
\cite[\S 12]{Tre2020}. Define
\[ \Pi_r(z) = \prod_{i=0}^{r-1} (z - p_i).\]
According to \cite[Thm. 11.1]{Tre2020}, the polynomial interpolant at the snapshot points can be written as 
\[ \vec{p}(p_*) = \frac{1}{2\pi i}\int_{\Gamma}\frac{\vec{x}(z)(\Pi_r(z)-\Pi_r(p_*))}{\Pi_r(z)(z-p_*)}dz\]
where $\Gamma$ is the boundary of the Bernstein ellipse in \Cref{assumption}.
Note that $\vec{p}(p_*) \in \mbox{span}(X)$. The error $ \vec{x}(p_*) -\vec{p}(p_*)$ satisfies
\[ \vec{x}(p_*) -\vec{p}(p_*) = \frac{1}{2\pi i}\int_{\Gamma}\frac{\vec{x}(z)\Pi_r(p_*)}{\Pi_r(z)(z-p_*)}dz.\]
 It follows that
\begin{align*} \bigl\| (I-XX^{\dagger})\vec{x}(p_*) \bigr\|_2 &\le   \bigl\| \vec{x}(p_*) - \vec{p}(p_*)  \bigr\|_2  \\
&\le \frac{1}{2\pi}\kappa_{\rho}\oint_{\Gamma}\frac{|\Pi_r(p_*)|}{|z-p_*||\Pi_r(z)|}|dz|
\end{align*}
We then have to bound
\begin{equation}\label{eq: Hermite}
\oint_{\Gamma}\frac{|\Pi_r(p_*)|}{|z-p_*||\Pi_r(z)|}|dz|.
\end{equation}
Using the logarithmic potential theory in \cite{Saf2010}, we know that, if $z \notin [-1,1]$,
\[ \frac1r \log\frac{1}{|\Pi_r(z)|} = U^{\mu_r}(z)\]
which implies, as $\mu_r \overset{\ast}{\rightharpoonup} \mu$, that
\[ \lim_{r \rightarrow \infty} \frac1r \log\frac{1}{|\Pi_r(z)|} = U^{\mu}(z)\]
or equivalently
\[ \lim_{r \rightarrow \infty}  |\Pi_r(z)|^{\frac1r} = \exp(-U^{\mu}(z)).\]
This limit holds uniformly on compact sets outside $[-1,1]$.
If $p_* \in [-1,1]$, the situation is more delicate as $U^{\mu_r}(z)$ is not continuous on $[-1,1]$. For this purpose, we use \cite[Lem 1.6]{Saf2010} to guarantee that
\[ \liminf_{r \rightarrow \infty}\frac1r \log\frac{1}{|\Pi_r(z)|} \geq U^{\mu}(z) \]
\[ \limsup_{r \rightarrow \infty} {|\Pi_r(z)|^{\frac1r}} \leq \exp(-U^{\mu}(z)). \]
This implies that, given an $\varepsilon >0$, there exists $r_0>0$ such that for all $r \geq r_0$
\[\oint_{\Gamma}\frac{|\Pi_r(p_*)|}{|z-p_*||\Pi_r(z)|}|dz| \leq \frac{1}{dist(\Gamma,p_*)}\oint_{\Gamma}\exp( -r(U^{\mu_r}(p_*) - U^{\mu_r}(z))|dz| \leq\]
\[ \leq \frac{1}{dist(\Gamma,p_*)}|\Gamma|\exp( -r(U^{\mu}(p_*) - \max_{z \in \Gamma}U^{\mu}(z) - 2\varepsilon)).\]
Under our assumption, we can find an $\varepsilon_0$ such that for all $\varepsilon \leq \varepsilon_0$
\[ U^{\mu}(p_*) - \max_{z \in \Gamma}U^{\mu}(z) -2\varepsilon > 0. \]
This concludes the proof.
\end{proof}

 This asymptotic rate was already derived in \cite[Thm 12.1]{Tre2020} in the context of polynomial approximation for uniformly distributed points.
Note that the assumption 
\[ \max_{z \in \partial E_{\rho}}U^{\mu}(z)  <U^{\mu_r}(p_*) \] is easily satisfied if $p_*$ is close to $[-1,1]$ and $\rho$ is large enough, as $U^{\mu}(z)$ decays to $-\infty$ as fast as $2\log\frac{1}{dist(z,[-1,1])}$.
Furthermore, we can give precise rates of convergence using this result.

\begin{theorem}\label{thm: asymptoticrate}
    Consider a unit positive Borel measure on $[-1,1]$. Define the snapshot points $p_0, \dots, p_{r-1}$ such that $p_i \neq p_j$ and the measure
    \[ \mu_r = \frac1r\sum_{i = 0}^{r-1} \delta(p_i)\]
    converges to $\mu$ in the weak-star topology. Let $p_*$ be a new parameter value (inside or outside $[-1,1]$). Then, asymptotically, the convergence is rate is given by the following expressions:
    \begin{itemize}
        \item The measure $\mu$ is the arcsine distribution $ \text{d}\mu(x) = \frac{\text{d}x}{\pi\sqrt{1-x^2}}$
        and the snapshot points $\{p_i\}_{i=0}^{r-1}$ are the Chebyshev points, and $p_* \in [-1,1]$, then
        \[ \limsup_{r \rightarrow \infty}\bigl\| (I-XX^{\dagger})\vec{x}(p_*) \bigr\|_2^{1/r} \le \rho^{-1}\]
        \item  The measure $\mu$ is the arcsine distribution 
        and the snapshot points $\{p_i\}_{i=0}^{r-1}$ are the Chebyshev points, and $p_* \notin [-1,1]$, then
        \[ \limsup_{r \rightarrow \infty}\bigl\| (I-XX^{\dagger})\vec{x}(p_*) \bigr\|^{1/r} \le \delta\]
        where
        \[ \delta = \frac{|p_*| + \sqrt{p_*^2 - 1}}{\rho}.\]
        \item Assuming $\rho$ is large enough, the measure $\mu$ is the uniform distribution 
        and the snapshot points $\{p_i\}_{i=0}^{r-1}$ are the equidistributed points, and $p_* \in [-1,1]$, then
        \[  \limsup_{r \rightarrow \infty}\bigl\| (I-XX^{\dagger})\vec{x}(p_*) \bigr\|_2^{1/r} \le \left(  \frac{2}{e \rho} (p_*+1)^{\frac{p_*+1}{2}}(1-p_*)^{\frac{1-p_*}{2}}\right)\le\frac{4}{e\rho}\]
        \item Assuming $\rho$ is large enough, the measure $\mu$ is the uniform distribution 
        and the snapshot points $\{p_i\}_{i=0}^{r-1}$ are the equidistributed points, and $p_* >1$, then
        \[  \limsup_{r \rightarrow \infty}\bigl\| (I-XX^{\dagger})\vec{x}(p_*) \bigr\|^{1/r} \le \left(  \frac{2}{e \rho} \frac{(p_*+1)^{\frac{p_*+1}{2}}}{(p_*-1)^{\frac{p_*-1}{2}}}\right)\]
    \end{itemize}
\end{theorem}

\begin{proof}
    The first two points are already proven in \Cref{thm: convergencecheb} in the non-asymptotic case. An alternative proof can be derived as the logarithmic potential for the arcsine measure is \cite[Def. 3.4]{Saf2010}
    \[ U^{\mu}(z) = \log(2) - \log|z + \sqrt{z^2 - 1}|\]
    for $z \notin [-1,1]$ while it is $\log(2)$ if $z \in [-1,1]$. It follows that
    \[ \exp(-r(U^{\mu}(p_*) - \max_{z \in \Gamma} U^{\mu}(z)) = \exp(-r(U^{\mu}(p_*) - \log(2) + \log(\rho)))\]
    where we have used that $U^{\mu}(z) = \log(2) - \log(\rho)$ for $z \in \Gamma$ as $\partial E_{\rho}$ is a level curve for $U^{\mu}(z)$. This is equal to $\rho^{-r}$ if $p_* \in [-1,1]$, while, if $p_* \notin [-1,1]$, it is equal to
    \[\exp(-r( - \log(|p_*| + \sqrt{p_*^2 - 1}) + \log(\rho)) \]
    If $\mu$ is the uniform measure, we assume that $\rho$ is large enough so that we can use 
    \[ \max_{z \in \Gamma} U^{\mu}(z) \approx -\log(\rho/2).\]
    On the other hand, $U^{\mu}(z)$ is not constant over $[-1,1]$ and it is equal to
     \[ U^{\mu}(z) = -\frac{1}{2}\int_{-1}^1 \log(|z - x|)dx = 1 - \frac{1}{2} \Big[ (z + 1)\log(z + 1) + (1-z)\log(1-z) \Big].\]
    The minimum potential is achieved at $1$
    \[ \min_{z \in [-1,1]} U^{\mu}(z) = U^{\mu}(1) =-\frac{1}{2}\int_{-1}^{1} \log|1-x|dx = 1-\log(2).\]
    From this, we get the formula above and the upper bound
    \[  \exp(-r(U^{\mu}(p_*) - \max_{z \in \Gamma} U^{\mu}(z)) \le \exp(-r(1-\log(2) +\log(\rho/2))) = e^{-r}2^r\rho^{-r}2^r = \left(\frac{e\rho}{4}\right)^{-r}.\]
    If $p_* \notin [-1,1]$, we can explicitly compute $U^{\mu}(p_*)$ (assuming $p_* >1$, $p_* < -1$ is the same)
    \[ U^{\mu}(p_*) = -\frac{1}{2}\int_{-1}^1 \log(p_* - x)dx = 1 - \frac{1}{2} \Big[ (p_* + 1)\log(p_* + 1) - (p_* - 1)\log(p_* - 1) \Big]\]
    from which
     \begin{align*}
    &\exp(-r(U^{\mu}(p_*) - \max_{z \in \Gamma} U^{\mu}(z)) = \\
      &=\exp\left(-r(1 - \frac{1}{2} \Big[ (p_* + 1)\log(p_* + 1) - (p_* - 1)\log(p_* - 1) \Big] +\log(\rho/2))\right)
      \end{align*}
     which concludes the proof.
    
\end{proof}
% =========================================================
% PLOT 1 & 2: p_r = 1.008765 (Outside [-1,1])
% =========================================================
\begin{figure}[ht]
    \centering
    % Subfigure 1: Non-asymptotic
    \begin{subfigure}[t]{0.45\textwidth}
    \vspace{0pt}
        \centering
        \begin{tikzpicture}
            \begin{semilogyaxis}[
                width=\linewidth,
                title={Non-asymp. bound ($p_*= 1.008765$)},
                xlabel={Number of snapshot points ($r$)},
                ylabel={$\bigl\| (I-XX^{\dagger})\vec{x}(p_*) \bigr\|_2$},
                ymin=1e-16, ymax=10,
                grid=major,
                legend style={
                    at={(0.5,-0.35)},
                    anchor=north,
                    legend columns=2,
                    font=\small,
                    cells={anchor=west},
                    inner xsep=1.5pt,
                    inner ysep=1.5pt,
                    nodes={scale=0.75, transform shape}
                },
                mark options={solid}
            ]
            \addplot[color=blue, mark=square*, thick] 
                table[x expr=\coordindex+4, y=err_cheb] {theory/1000_3.1_1.008765.tex};
            \addlegendentry{Error $X_{\text{cheb}}$}
            
            \addplot[color=red, mark=*, thick] 
                table[x expr=\coordindex+4, y=err_equi] {theory/1000_3.1_1.008765.tex};
            \addlegendentry{Error $X_{\text{equi}}$}
            
            \addplot[color=black, dashed, thick, mark=none] 
                table[x expr=\coordindex+4, y=rate_cheb] {theory/1000_3.1_1.008765.tex};
            \addlegendentry{Fit Cheb $O(e^{-c_1 M})$}
            
            \addplot[color=magenta, dashed, very thick, mark=none] 
                table[x expr=\coordindex+4, y=rate_equi_sqr] {theory/1000_3.1_1.008765.tex};
            \addlegendentry{Fit Equi $O(e^{-c_2 \sqrt{M}})$}
            
            \addplot[color=black, dashdotted, thick, mark=none] 
                table[x expr=\coordindex+4, y=rate_equi_asym] {theory/1000_3.1_1.008765.tex};
            \addlegendentry{Fit Equi $O(e^{-c_3 M})$}
            
            \addplot[color=cyan, dotted, ultra thick, mark=none] 
                table[x expr=\coordindex+4, y=th_cheb_scaled_discrete] {theory/1000_3.1_1.008765.tex};
            \addlegendentry{Bound Cheb}
            
            \addplot[color=yellow!80!black, dotted, ultra thick, mark=none] 
                table[x expr=\coordindex+4, y=th_equi_scaled_discrete] {theory/1000_3.1_1.008765.tex};
            \addlegendentry{Bound Equi}
            \end{semilogyaxis}
        \end{tikzpicture}
    \end{subfigure}
    \hfill
    % Subfigure 2: Asymptotic
    \begin{subfigure}[t]{0.45\textwidth}
    \vspace{0pt}
        \centering
        \begin{tikzpicture}
            \begin{semilogyaxis}[
                width=\linewidth,
                title={Asymp. bound ($p_*= 1.008765$)},
                xlabel={Number of snapshot points ($r$)},
                ymin=1e-16, ymax=10,
                grid=major,
                legend style={
                    at={(0.5,-0.35)},
                    anchor=north,
                    legend columns=2,
                    font=\small,
                    cells={anchor=west},
                    inner xsep=1.5pt,
                    inner ysep=1.5pt,
                    nodes={scale=0.75, transform shape}
                }
            ]
            \addplot[color=blue, mark=square*, thick] 
                table[x expr=\coordindex+4, y=err_cheb] {theory/1000_3.1_1.008765.tex};
            \addlegendentry{Error $X_{\text{cheb}}$}
            
            \addplot[color=red, mark=*, thick] 
                table[x expr=\coordindex+4, y=err_equi] {theory/1000_3.1_1.008765.tex};
            \addlegendentry{Error $X_{\text{equi}}$}
            
            \addplot[color=cyan, dashed, very thick, mark=none] 
                table[x expr=\coordindex+4, y=th_cheb_scaled_apriori] {theory/1000_3.1_1.008765.tex};
            \addlegendentry{Asymp. Cheb}
            
            \addplot[color=yellow!80!black, dashed, very thick, mark=none] 
                table[x expr=\coordindex+4, y=th_equi_scaled_apriori] {theory/1000_3.1_1.008765.tex};
            \addlegendentry{Asymp. Equi}
            \end{semilogyaxis}
        \end{tikzpicture}
    \end{subfigure}
    \caption{Non-asymptotic (left) and asymptotic (right) error bounds for $p_* = 1.008765$ (outside $[-1,1]$). The setting is the same as for \Cref{fig: singdecay}. In blue and red we plot the real error (the norm of the projection of $\vec{x}(p_*)$ onto $\mbox{span}(X)^{\perp}$). On the left, we plot interpolations of the form $e^{-cr}$ and $e^{-c\sqrt{r}}$ and the bounds using the discrete measure $\mu_r$. On the right, we plot the bounds from \Cref{thm: asymptoticrate} using the limit measure $\mu$. We estimate $\max_{z \in \Gamma} U(z)$ with $-\log(\rho/2)$, where $\rho = 3.1$ is estimated from the decay rate in \Cref{fig: singdecay}. The constants in front of the bounds are chosen to best fit the real error.}
    \label{fig:error_bounds_outside}
\end{figure}
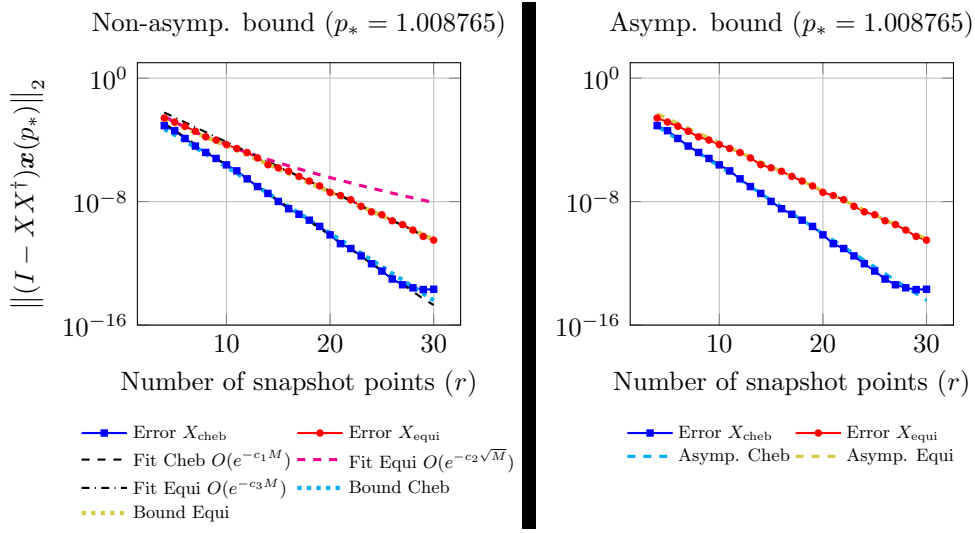

% =========================================================
% PLOT 3 & 4: p_r = 0.8765 (Inside [-1,1])
% =========================================================
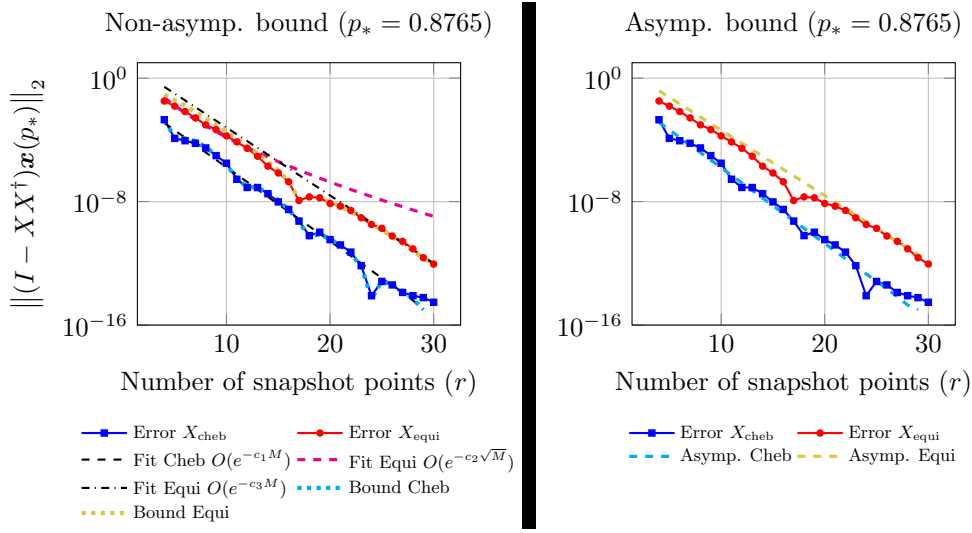
\begin{figure}[ht]
    \centering
    % Subfigure 1: Non-asymptotic
    \begin{subfigure}[t]{0.45\textwidth}
    \vspace{0pt}
        \centering
        \begin{tikzpicture}
            \begin{semilogyaxis}[
                width=\linewidth,
                title={Non-asymp. bound ($p_*= 0.8765$)},
                xlabel={Number of snapshot points ($r$)},
                ylabel={$\bigl\| (I-XX^{\dagger})\vec{x}(p_*) \bigr\|_2$},
                ymin=1e-16, ymax=10,
                grid=major,
                legend style={
                    at={(0.5,-0.35)},
                    anchor=north,
                    legend columns=2,
                    font=\small,
                    cells={anchor=west},
                    inner xsep=1.5pt,
                    inner ysep=1.5pt,
                    nodes={scale=0.75, transform shape}
                },
                mark options={solid}
            ]
            \addplot[color=blue, mark=square*, thick] 
                table[x expr=\coordindex+4, y=err_cheb] {theory/1000_3.1_0.8765.tex};
            \addlegendentry{Error $X_{\text{cheb}}$}
            
            \addplot[color=red, mark=*, thick] 
                table[x expr=\coordindex+4, y=err_equi] {theory/1000_3.1_0.8765.tex};
            \addlegendentry{Error $X_{\text{equi}}$}
            
            \addplot[color=black, dashed, thick, mark=none] 
                table[x expr=\coordindex+4, y=rate_cheb] {theory/1000_3.1_0.8765.tex};
            \addlegendentry{Fit Cheb $O(e^{-c_1 M})$}
            
            \addplot[color=magenta, dashed, very thick, mark=none] 
                table[x expr=\coordindex+4, y=rate_equi_sqr] {theory/1000_3.1_0.8765.tex};
            \addlegendentry{Fit Equi $O(e^{-c_2 \sqrt{M}})$}
            
            \addplot[color=black, dashdotted, thick, mark=none] 
                table[x expr=\coordindex+4, y=rate_equi_asym] {theory/1000_3.1_0.8765.tex};
            \addlegendentry{Fit Equi $O(e^{-c_3 M})$}
            
            \addplot[color=cyan, dotted, ultra thick, mark=none] 
                table[x expr=\coordindex+4, y=th_cheb_scaled_discrete] {theory/1000_3.1_0.8765.tex};
            \addlegendentry{Bound Cheb}
            
            \addplot[color=yellow!80!black, dotted, ultra thick, mark=none] 
                table[x expr=\coordindex+4, y=th_equi_scaled_discrete] {theory/1000_3.1_0.8765.tex};
            \addlegendentry{Bound Equi}
            \end{semilogyaxis}
        \end{tikzpicture}
    \end{subfigure}
    \hfill
    % Subfigure 2: Asymptotic
    \begin{subfigure}[t]{0.45\textwidth}
    \vspace{0pt}
        \centering
        \begin{tikzpicture}
            \begin{semilogyaxis}[
                width=\linewidth,
                title={Asymp. bound ($p_*= 0.8765$)},
                xlabel={Number of snapshot points ($r$)},
                ymin=1e-16, ymax=10,
                grid=major,
                legend style={
                    at={(0.5,-0.35)},
                    anchor=north,
                    legend columns=2,
                    font=\small,
                    cells={anchor=west},
                    inner xsep=1.5pt,
                    inner ysep=1.5pt,
                    nodes={scale=0.75, transform shape}
                }
            ]
            \addplot[color=blue, mark=square*, thick] 
                table[x expr=\coordindex+4, y=err_cheb] {theory/1000_3.1_0.8765.tex};
            \addlegendentry{Error $X_{\text{cheb}}$}
            
            \addplot[color=red, mark=*, thick] 
                table[x expr=\coordindex+4, y=err_equi] {theory/1000_3.1_0.8765.tex};
            \addlegendentry{Error $X_{\text{equi}}$}
            
            \addplot[color=cyan, dashed, very thick, mark=none] 
                table[x expr=\coordindex+4, y=th_cheb_scaled_apriori] {theory/1000_3.1_0.8765.tex};
            \addlegendentry{Asymp. Cheb}
            
            \addplot[color=yellow!80!black, dashed, very thick, mark=none] 
                table[x expr=\coordindex+4, y=th_equi_scaled_apriori] {theory/1000_3.1_0.8765.tex};
            \addlegendentry{Asymp. Equi}
            \end{semilogyaxis}
        \end{tikzpicture}
    \end{subfigure}
    \caption{Non-asymptotic (left) and asymptotic (right) error bounds for $p_* = 0.8765$ (inside $[-1,1]$). Similar to \Cref{fig:error_bounds_outside}, we plot the real error against the bounds calculated with the discrete measure $\mu_r$ (left) and the limit measure $\mu$ (right). Parameter estimation follows the same logic, with $\rho = 3.1$.}
    \label{fig:error_bounds_inside}
\end{figure}

An extension to the case where $p$ lives in a high-dimensional space may be analyzed by using the analogous tools for high-dimensional potential theory and high-dimensional polynomial approximation \cite{Lev2006}.

 In practice, finding the projection of $\vec{x}(p_*)$ on $\operatorname{span}(X)$ is difficult  unless we have access to $\vec{x}(p_*)$ itself. 
Classically, if a good subspace is at disposal, $\vec{x}(p_*)$ is approximated using the Rayleigh-Ritz method \cite{Saa2011}, i.e. we solve

\[ M(p) = \operatorname{argmin}_{M \in \C^{r \times r}} \|A(p)X-XM \|_F\]
we compute the corresponding eigenpair $(\Tilde{\lambda}(p), \vec{y}(p))$ of $M(p)$
\[ M(p)\vec{y}(p) = \Tilde{\lambda}(p)\vec{y}(p)\]
and we return $(\Tilde{\lambda}(p), X\vec{y}(p))$. Note that, even if Rayleigh-Ritz may work well in many cases, it can break down in specific examples \cite{Ste2001}.

\subsection{Theoretical results on subsampling for Rayleigh-Ritz}

In this section, we present some results on the accuracy of the solution of a least square problem solved using subsample. Define
\[ M(p) = \operatorname{argmin}_{M \in \C^{r \times r}} \| A(p)X-XM\|_F\]
and 
\[ \widehat{M}(p) = \operatorname{argmin}_{M \in \C^{r \times r}} \|S( A(p)X-XM)\|_F\]
We exploit the following theorem. 

\begin{lemma}[{\cite[Lemma 5.1]{Jon2025}}]\label{thm: boundsub}
Let $B \in \mathbb{R}^{n \times r}$, $\vec{b} \in \mathbb{R}^n$ with $n > r$, and let 
$\vec{c}, \hat{\vec{c}}$ denote the solutions for the least-squares problems
\[
\min_{\vec{c}} \|B\vec{c} - \vec{b}\|_2 
\qquad\text{and}\qquad
\min_{\hat{\vec{c}}} \|S(B\hat{\vec{c}} - \vec{b})\|_2,
\]
where $S \in \mathbb{R}^{s \times n}$ with $r \le s \le n$.
Then we have
\begin{equation}
1\leq\frac{\|B\hat{\vec{c}} - \vec{b}\|_2}{\|B\vec{c} - \vec{b}\|_2}
\le 
\frac{\|S\|_2}{\sigma_{\min}(SU)},
\end{equation}
where $B = UH$ is the polar decomposition. We also have
\begin{equation}
1\leq\frac{\|B\hat{\vec{c}} - \vec{b}\|_2}{\|B\vec{c} - \vec{b}\|_2}
\le 
\kappa_2(\widetilde{S}U),
\end{equation}
where $\widetilde{U}\widetilde{H} = [B\ \vec{b}] \in \mathbb{R}^{n \times (r+1)}$ is the polar decomposition, and 
$\kappa_2(\widetilde{S}U) = \frac{\sigma_{\max}(\widetilde{S}U)}{\sigma_{\min}(\widetilde{S}U)}.$
\end{lemma}
We can follow the proof of \cite[Lemma 5.1]{Jon2025} to prove the following theorem.
\begin{lemma}\label{thm: boundsubsrr}
\begin{equation}\label{eqn:firstbound}
\frac{\|X\widehat{M}(p) - A(p)X\|_F}{\|XM(p) - A(p)X\|_F}
\le 
\frac{\|S\|_2}{\sigma_{\min}(SU)},
\end{equation}
where $X = UH$ is the polar decomposition. We also have
\begin{equation}
\frac{\|X\widehat{M}(p) - A(p)X\|_F}{\|XM(p) - A(p)X\|_F}
\le 
\kappa_2(S\widetilde{U}),
\end{equation}
where $\widetilde{U}\widetilde{H} = [X \ A(p)X] \in \C^{n \times 2r}$ is the polar decomposition.
\end{lemma}

\begin{proof}
The proof is the same as \cite[Lemma 5.1]{Jon2025}, where we treat $\vec{b}$ as a matrix rather than a vector.
\end{proof}
Using \Cref{thm: boundsubsrr} and the result in \cite[\S 6.6]{Nak2024}, we get that
\begin{equation}\label{eq: bound}
\begin{aligned}
 \| M(p) - \widehat{M}(p)\|_F &\leq \frac{1}{\sigma_{\min}(X)}\| X(M - \widehat{M})\|_F \\
 &= \frac{1}{\sigma_{\min}(X)}\|(A(p)X-X\widehat{M}(p)) - (A(p)X- XM(p))\|_F \\
& \leq \frac{1}{\sigma_{\min}(X)}(\|(A(p)X-X\widehat{M}(p))\|_F +\|(A(p)X- XM(p))\|_F) \\
& \leq\frac{1}{\sigma_{\min}(X)}\left(\frac{\|S\|_2}{\sigma_{\min}(SU)}  + 1\right)\|A(p)X-XM\|_F  \\
&= \kappa_2(X)\left(\frac{\|S\|_2}{\sigma_{\min}(SU)}  + 1\right)\|(I-QQ^*)A(p)Q\|_F
\end{aligned}
\end{equation}
where $Q$ are the left singular vectors of $X$. 
This guarantees that, if we choose $S$ well and $\kappa_2(X)$ is not too large (for example if we use POD and replace $X$ with its first left singular vectors) and $Q$ is almost invariant for $A(p)$, then
\[ M(p) \approx \widehat{M}(p).\]

If the matrices $A(p)$ are Hermitian, we may want to output real eigenvalues. As $\widehat{M}(p)$ is not Hermitian in general, we do not always have real eigenvalues. Some solutions are listed in \cite{Nak2024}, for example returning the real part of the computed eigenvalues and eigenvectors. 
Note that, if we orthogonalize $X$ and use an orthonormal matrix $Q$ in the least-square problem instead of $X$, then 
\[ M(p) = Q^*A(p)Q\]
is Hermitian and we can make the substitution
\[ \widehat{M}(p) \leftarrow \frac{\widehat{M}(p) + \widehat{M}(p)^*}{2}\]
without making the solution worse. In fact

\[ \| M(p) - \widehat{M}(p)\|_F \geq \left\| M(p) -  \frac{\widehat{M}(p) + \widehat{M}(p)^*}{2}\right\|_F\]
as 
\[ M(p) \perp \frac{\widehat{M}(p) - \widehat{M}(p)^*}{2}.\]

\begin{remark}[\textbf{Varying $p$}]
Note that our problem has a crucial difference from \cite[\S 5.2]{Jon2025}, that is, the bound in \eqref{eqn:firstbound} and in \eqref{eq: bound}, does \textit{not} depend on $A(p)$ but only on $X$. This suggests that we can select rows from $X$ and, if $X$ is not changed, we can keep the same set of rows for different values of the parameter. This is crucial as it guarantees that the only error in the algorithm comes from a poor performance of the snapshot matrix when $p$ varies, and not by subsampling, as long as the subsampling matrix $S$ was chosen well at the beginning.
\end{remark}
 \section{Main algorithm}\label{sec: algo}

 In this section, we outline our algorithm. We assume we have a parametrized sequence of matrices $A(p) \in \C^{n \times n}$ where $p \in \Omega$, where $\Omega$ is a domain in $\C$ and we assume we want to compute $m$ eigenpairs for each value of the parameter. 
 The algorithm is divided into two parts: an \textit{offline} phase and an \textit{online} phase.
 After the \textit{offline} phase is carried out, our goal in the \textit{online} phase is to quickly compute a set of eigenpairs
 \[ \{ (\lambda_1(p_*), \vec{x}_1(p_*)), \dots, (\lambda_m(p_*), \vec{x}_m(p_*)) \}\]
 for an unseen parameter value $p_*$ or for a large number of unseen parameters.

 The \textit{offline} phase is completely analogous to the setting of \textit{reduced basis} and it is devoted to the computation of the snapshot matrix. First, we select a subset $\{ p_i\}_{i=0}^{r-1}$ of parameters and we compute the eigenpairs 
  \[ \{ (\lambda_1(p_i), \vec{x}_1(p_i)), \dots, (\lambda_m(p_i), \vec{x}_m(p_i)) \} \quad i = 0,\dots, r-1\]
  using a standard algorithm. Then, we form the snapshot matrix 
  \[ X = [\vec{x}_1(p_0), \dots \vec{x}_m(p_0), \dots,\vec{x}_1(p_{r-1}), \dots , \vec{x}_m(p_{r-1})].\]

  Optionally, we can substitute $X$ with $Q$ where $Q$ are the first left singular vectors. As we know that $X$ has quickly decaying singular values by \Cref{thm:singdecay}, we can select only the left singular vectors associated with singular values greater than a prescribed quantity, as done in POD.

  After the \textit{offline} phase, we have to compute the matrix $S$ (or equivalently, the row indices). One option (which will always be our choice) is to directly select from $X$ (or $Q)$ using one of the classical row/column selection methods (see \Cref{subsub:substrat}, \cite{Hal2011}, \cite[\S 2.1.2]{Don2023}). Another more conservative option is to select from 
  \[ [X , A(p)X] \in \C^{n \times 2r}\]
  where $p$ is the first parameter value we see in the online phase.
  Note that, in this second scenario, we usually obtain more rows, as we are subsampling from a matrix with double the number of columns. This second strategy is more expensive but has some advantages, for example we can compute a cheap estimate of the residual; see \Cref{sub:residual}. Furthermore, in this second scenario, we need to access $A(p)$ and compute the matrix-vector products of $A(p)X$. According to \cite[\S 5.2]{Jon2025}, we can select the row indices for each snapshot point, and, for a new parameter $p$, we employ the row indices of the nearest snapshot point. This strategy does not require accessing $A(p)$ for new values of the parameter.

  In the \textit{online} phase we choose a new parameter $p_*$ and solve

  \[ \widehat{M}(p_*) = \operatorname{argmin}_{M \in \C^{r \times r}}\| SA(p_*)X - SXM\|_F.\]
  Note that $SX$ can be precomputed, while $SA(p_*)X$ can be computed as $(SA(p_*))X$, possibly without looking at the entire matrix $A(p_*)$. Once we have found $\widehat{M}(p_*) \in \C^{r \times r}$, we compute the corresponding eigenpairs
   \[ \{ (\widehat{\lambda}_1(p_*), \vec{y}_1(p_*)), \dots, (\widehat{\lambda}_m(p_*), \vec{y}_m(p_*)) \}\]
  of $\widehat{M}(p_*)$ using a dense algorithm. If $A(p)$ is Hermitian, we may take the real part of these eigenvalues. 
  
  At the end, we return

  \[  \{ (\widehat{\lambda}_1(p_*), X\vec{y}_1(p_*)), \dots, (\widehat{\lambda}_m(p_*), X\vec{y}_m(p_*)) \}\]

  The online phase can be carried out for a large number of new parameters using the same $S$ and $X$. On the other hand, it is advisable to update both $X$ and $S$ every fixed number of steps, or if the estimate residual norm is large (see \Cref{sec: imple} for a residual estimation), especially in the time-dependent setting, where the new parameter values are further and further away from the initial snapshot points; see \Cref{sub:residual} for the residual estimation and the update of $S$. If $X$ is to be updated, we refine the last computed eigenvectors using some iterations of a Krylov method (for example \texttt{eigs} in \texttt{MATLAB}) and we add these new vectors to $X$. The pseudocode is outlined in \Cref{alg}. We name our algorithm \texttt{SubApSnapEig}, following the naming convention introduced in \cite{Jon2025}.

 \begin{algorithm}
  \caption{\texttt{SubApSnapEig} for solving a sequence of eigenvalue problems}\label{alg}
  \begin{algorithmic}[1]
  \Require Snapshot points $p_0, \dots ,p_{r-1} \in \Omega$, new parameter value $p_* \in \Omega$ or a set of $\ell$ new parameters
      $\vec{p}_* \subset \Omega$, subsampling strategy, number $m$ of desired eigenpairs, tolerance \texttt{tol} and update frequency \texttt{freq}.
  \Ensure $m$ eigenpairs for parameter value $p_*$ or for all the parameter values in $\vec{p}_*$.
  
    \State \makebox[0pt][l]{\hspace{36em}%
      \smash{\raisebox{-4ex}{% Regola questo valore per centrare verticalmente la graffa
        $\left.\begin{array}{@{}c@{}}
        \\{}\\{}\\{}\\{}\\{}\\{}
        \end{array}\right\}
        \begin{tabular}{l}
          Offline phase
        \end{tabular}$%
      }}%
    }%
    Find $m$ eigenpairs of $A(p_i), \ i=0,\dots ,r-1$ and form  
   $X = [\vec{x}_1(p_0), \dots \vec{x}_m(p_0), \dots,\vec{x}_1(p_{r-1}), \dots , \vec{x}_m(p_{r-1})]$ 

    \State (Optional, recommended when $mr \gg 1$) Find  an (approximate)
    orthonormal matrix $Q$ that spans the range of $X$ (for example with (randomized) SVD) 
    \State Use subsampling strategy on $Q$ (or on $[Q, A(p)Q]$ for $p$ taking the value of the snapshot points) and 
    obtain the matrix $S$. 

    \For {$p_* \in \boldsymbol{p}_*$}
    
    \State \makebox[0pt][l]{\hspace{34.5em}%
      \smash{\raisebox{-8ex}{% Regola questo valore per centrare verticalmente la graffa
        $\left.\begin{array}{@{}c@{}}
        \\{}\\{}\\{}\\{}\\{}\\{}\\{}\\{}
        \end{array}\right\}
        \begin{tabular}{l}
          Online phase
        \end{tabular}$%
      }}%
    }%
    Solve the least square problem $\widehat{M}(p_*) = \operatorname{argmin}_{M \in \C^{r \times r}}\| SA(p_*)Q - SQM\|_F$ and compute $m$ corresponding eigenpairs of $\widehat{M}(p_*)$. 
    \State If \texttt{freq} parameter value has been processed since last update, update $X$ with the new refined eigenvectors and update $S$.
    \State (Optional) If leverage score sampling with oversampling is employed, estimate the residual. If residual norm is larger than \texttt{tol}, update $X$ and $S$.
    \State \Return $\{(\widehat{\lambda}_i(p_*), X\vec{y}_i(p_*))\}_{i=1}^m.$
    \EndFor

\end{algorithmic}
\end{algorithm}

Assuming that the number of rows selected, $s$, satisfies $s = \mathcal{O}(mr)$ (we take  $s = 4mr$ in our experiments) and that we subsample from $X$, the complexity of the algorithm breaks down as follows: in the \textit{offline} phase, the computation of $X$ costs $\mathcal{O}(n^3r)$ to which we may add the orthogonalization of $X$, whose complexity is $\mathcal{O}(n(mr)^2)$ and the subsampling $\mathcal{O}(nmr)$, assuming a row of $A(p_*)$ can be computed in $\mathcal{O}(n)$ operations. In the \textit{online} phase, it costs $\mathcal{O}(n(mr)^2)$ to compute $(SA(p_*))X$, while the least square problem requires $\mathcal{O}((mr)^3)$ operations and computing the eigenpairs $\mathcal{O}((mr)^3)$ operations.
In conclusion, the complexity for the \textit{offline} phase is 
\[ \mathcal{O}(n^3r) + \mathcal{O}(n(mr)^2) + \mathcal{O}(nmr) = \mathcal{O}(n^3r + n(mr)^2)\]
while the complexity for the \textit{online} phase is 
\[ \mathcal{O}(n(mr)^2\ell) + \mathcal{O}((mr)^3\ell) + \mathcal{O}((mr)^3\ell) = \mathcal{O}(n(mr)^2\ell+(mr)^3\ell)\]
where $\ell$ is the number of new parameters we want to process.

 \section{Implementation and residual estimation}\label{sec: imple}

 In this section, we discuss implementation details for our algorithm.

 \subsubsection*{Computation of $SA(p_*)X$}
 As already mentioned, we compute $SA(p_*)X$ as
 $(SA(p_*))X$, i.e we first subsample from $A(p_*)$ and then we compute the product. Note that this approach is much more convenient in numerous settings. For example, if we can access $A(p)$ entry-wise and computing single entries of $A(p)$ is expensive, computing $SA(p)$ is more efficient as we do not need to know all the entries of $A(p)$.

 \subsubsection*{Orthogonalisation and truncation of snapshot matrix $X$}

We can orthogonalize the matrix $X$ by taking a thin rank-revealing QR decomposition $X = QR$ or by taking the SVD, as in POD \cite{Ber1993, Cha2000} or the randomized SVD \cite{Gui2024}. This is particularly
effective when $mr \gg 1$ and the snapshot matrix $X$ is numerically rank deficient, as orthogonalising gives better numerical stability and better-conditioned least-squares problems.

\subsubsection*{Update of $X$ and $S$}

In order for the algorithm to perform well for a large set of parameter values, we may want to implement some update strategy for the snapshot matrix $X$ and consequently for $S$. We update $X$ in two cases: if we have not updated for a prescribed number of parameters (for time-dependent problems), or the estimate residual is large, see \Cref{sub:residual}. Note that the computed eigenvectors at any step belongs to the span of $X$, so in order to inject new information in the snapshot matrix, we have to refine the computed eigenvectors. We do this by running a few iterations of a Krylov method, in particular we run \texttt{MATLAB} \texttt{eigs} with start vector equal to the previously computed eigenvector. This is often a good estimate so we expect \texttt{eigs} to converge in a few iterations.

More importantly, to update the matrix $S$ we do not have to compute the  matrix $A(p)X$ (differently from the linear system case), as we choose the subsampling $S$ using only $X$ (or its left singular vectors $Q$).

\subsubsection*{Subsampling strategies}\label{subsub:substrat}
It is of crucial importance to select a subsample strategy that fits the problem we are tackling. Many determistic algorithms are available in the literature \cite{Don2023} for the column/row selection problem: LU with partial pivoting (LUPP), column-pivoted QR (CPQR), Osinsky's approach \cite{Osi2025}. Randomized methods are also available: leverage score sampling \cite{Dri2012}, determinantal point process \cite{Der2021} (expensive but strong theoretical guarantees), and adaptive randomized pivoting \cite{Cor2025}. Note that we usually subsample from $X$ only, another possible choice is to subsample from $[X, A(p)X]$. This requires accessing to $A(p)$ and to select a large number of rows, but guarantees stronger theoretical bounds. Note that if $A(p)$ is already accessed entirely to refine the eigenvector, then the multiplication $A(p)X$ only contributes to the computational time for the matrix-vector products, as $A(p)$ is already computed.

\subsection{Residual estimate and change of $S$}\label{sub:residual}

In this section, we provide a cheap residual estimate available if leverage score sampling with oversampling (i.e., $s>r$) is performed in the solution of the least square problem.
According to \cite[\S 9.6.3]{Mar2020}, we know that, if $S \in \R^{s\times n }$ is generated through leverage score sampling and every row is divided by $\sqrt{sp_i}$, where $p_i$ is the probability of selecting that row, then we have a subspace embedding property. In particular 

\begin{theorem}

If we subsample using leverage scores with oversampling from $Z \in \R^{n \times r}$
then 
\[ (1 +\varepsilon)\left\|\frac{1}{\sqrt
s}SZ\vec{x}\right\| \leq \|Z\vec{x}\| \leq (1 + \varepsilon)\left\|\frac{1}{\sqrt
s}SZ\vec{x}\right\| \quad \forall \vec{x} \in \C^r\]
if $S \in \R^{s \times n}$ 
\[ s \ge 2\varepsilon^{-2}r\log(2r).\]
\end{theorem}

 In our case, if we subsample from $[X, AX]$ with oversampling, and if $\vec{y}$ is an eigenvector of $\widehat{M}$
 \[ \widehat{M}\vec{y} = \lambda \vec{y}\]
 then
 \[ \|S(AX\vec{y} - \lambda X\vec{y})\| =  \|S([X,AX] \begin{bmatrix}-\lambda \vec{y}\\\vec{y}\end{bmatrix})\|\]

 \[ (1 - \varepsilon)\|AX\vec{y}-\lambda X\vec{y}\|\le\|S([X,AX] \begin{bmatrix}-\lambda \vec{y}\\\vec{y}\end{bmatrix})\|  \leq (1 + \varepsilon)\|AX\vec{y}-\lambda X\vec{y}\|.\]
 It follows that 
\[\frac{\|S(A(p)X\vec{y}(p) - \widehat{\lambda}(p) X\vec{y}(p))\|}{\|X\vec{y}(p)\|}\]
is a good estimate of the actual error. Furthermore, $SA(p)X$ and $SX$ have already been computed to solve the least-squares problem. Accordingly, we only need to compute $(SA(p)X)\vec{y}(p)$ that costs $\mathcal{O}(sr)$, which is negligible.

In practice we would like to compute a relative error
\[ \frac{\|S(A(p)X\vec{y}(p) - \widehat{\lambda}(p) X\vec{y}(p))\|}{\|A(p)\|\|X\vec{y}(p)\|}.\]
As computing $\|A(p)\|$ requires access to the the entire matrix $A(p)$, every time we access $\|A(p)\|$ (for example in the snapshot points), we save its norm and for a new parameter $p_*$, we employ the norm of the closest snapshot point.

In practice, we subsample from $X \in \C^{n \times r}$ and not from $[X, A(p)X] \in \C^{n \times 2r}$. This method can still be used to get a rough estimate on the residual; if this estimate is large, it suggests that the snapshot $X$ or the subsampling $S$ need to be updated. If oversampling is not desirable, LUPP or CPQR are usually slightly faster, but without oversampling we have no information on the residual.

 \section{Related algorithms}\label{sec: related}

 There are various algorithms and strategies in the literature that exploit information from previous problems (in our case, eigenpairs already computed for the snapshot points). These include: Krylov subspace recycling \cite{Kil2006, Par2006}, rational Krylov methods \cite{Sim2010}, Taylor/ Chebyshev expansions \cite{Cor2025/2, Cor2023, Jar2022}, model order reduction through snapshots \cite{Mar2006, Gui2024}. One notable algorithm that exploits an idea similar to subsampling (besides snapshots) is DEIM \cite{Cha2010}.

 The idea is pretty similar,
 but in DEIM we 
 are able to evaluate exactly (a small number of entries of) the vectors that we want to approximate (the eigenvectors in our case). The setting of DEIM is as follows: given a parametrized vector $\vec{x}(p) \in \C^{n}$, where $p \in \Omega$, the goal is to find a subspace of dimension $r \ll n$ such that $\vec{x}(p)$ belongs to this subspace up to a small error. In other words, the goal is merely model order reduction, and no linear systems or eigenvalue problems are involved. More precisely, first a snapshot matrix is formed 
 \[ X = [\vec{x}(p_0),\dots,\vec{x}(p_{r-1)} ]\]
 is formed, and then the least square problem \[\| X\vec{c}(p) - \vec{x}(p)\|_2\]
 is solved using subsampling. The crucial assumption made here is that we are able to evaluate a single entry of $\vec{x}(p)$ relatively quickly
 (but not the entire $\vec{x}(p)$). This assumption is (usually) not satisfied if $\vec{x}(p)$ is chosen as an eigenvector of $A(p)$.

  \section{SVD and GEP}\label{sec: svd}

  The same algorithm can be applied for SVD. Instead of one snapshot matrix, we need two, $X_r$ and $X_{\ell}$, respectively for right and left singular vectors. Once we have these, we can solve the least-square problem

  \[ M(p) = \operatorname{argmin}_{M \in \C^{r \times r}}\|X_{\ell}M - A(p)X_r\|_F.\]

  If we want to sketch or subsample this problem, we can extract row indices from $X_{\ell}$. Again, we do not need to look at the entire $A(p)$. The remainder of the algorithm is very similar to eigenvalues; the only difference is that we are dealing with two snapshot matrices instead of one.
  Another variant is to just approximate $X_{\ell}^{\dagger}A(p)$ using sketching. We can do this by sketch-and-solve the least-square problem
  \begin{equation}\label{eq: lsqsvd}
 {\arg\min}_{M \in \C^{r \times n}}\|A(p) -X_{\ell}M\|_F.
 \end{equation}
  In this case, we can still select rows from $X_{\ell}$. 
  This second strategy is appealing because solving \eqref{eq: lsqsvd} is equivalent to the Rayleigh-Ritz method on $A(p)(A(p))^*$. The disadvantage of this second method is that the subsampled least-square problem 
  \[ {\arg\min}_{M \in \C^{r \times n}}\|SA(p) -SX_{\ell}M\| \]
  has dimension $s \times n$, which still depends on $n$ and $M \in \C^{r \times n}$. Even though the cost of the two subsampled least-squares is asymptotically the same ($\mathcal{O}(nr^2)$), computing the left singular vectors for the reduced matrix costs $\mathcal{O}(r^3)$ in the first method while it costs $\mathcal{O}(nr^2)$ in the second method.

  For the generalized eigenvalue problem, we have a sequence of pairs of matrices $(A(p), B(p))$ and we want to solve
  \[ A(p)\vec{x}(p) = \lambda(p)B(p)\vec{x}(p).\]
  In this case, the least-square to solve is
  \[ \widehat{M}(p) = \text{argmin}_{M \in \C^{r \times r}}\|S(B(p)XM - A(p)X)\|_F.\]
 In this case, we can select the rows from $B(p)X$ in the snapshot points and for a new value of the parameter we can employ the same indices of the nearest snapshot point. If we employ this heuristic, we do not have to access $B(p)$ for selecting the rows. In the case of time-dependent problems, the snapshot points get further and further from the current value of the parameter, so updating the set of rows requires accessing the entire $B(p)$.

\section{Numerical Experiments}\label{sec: experiments}

In this section, we provide some numerical experiments comparing our algorithm with a direct extension of \cite{Gui2024} applied to eigenvalue problems, named \texttt{MOR}, that builds the same snapshot matrix in the offline phase but it computes the projection of $A(p)$ in the online phase without using row subsampling. In our method we use two subsampling strategies for our row extraction: LUPP (referred to as \texttt{LU}) and Leverage scores (referred to as \texttt{Leverage}). Our algorithm \Cref{alg} is named \texttt{SubApSnapEig}.

For all our experiments, the initial eigenpairs for the snapshot points are computed using \texttt{eigs} in \texttt{MATLAB}, possibly employing a buffer (i.e. we create $X$ using more eigenpairs per parameter than the ones we want) to capture a richer initial subspace. We extract the row indices from the orthogonal basis $Q$ of the snapshot matrix $X$. The matrix $Q$ is subsequently used to solve the projected least-squares problem. 

 In our examples, we update the snapshot matrix $X$ only for the time-dependent scenario and only for \Cref{sub:eigcen}. When we updated $X$, we just add the refined eigenvectors to $X$ (which then has an increasing number of columns). The matrix $S$ is extracted from $X$ and is updated if and only if $X$ is updated.

For the plots, we compute the residual for the plots as follows:
\begin{equation}
    \frac{\|A(p)X\vec{y}(p) - \widehat{\lambda}(p)X\vec{y}(p)\|_F}{\|X\vec{y}(p)\|_F}
\end{equation}
where $(\widehat{\lambda}(p), \vec{y}(p))$ is the computed eigenpair of the projected matrix $\widehat{M}(p)$. For the eigenvalues error, we plot the relative error $\frac{|\lambda(p) - \widehat{\lambda}(p)|}{|\lambda(p)|}$, and for the eigenvectors error, we plot the subspace distance to the exact eigenspace using the \texttt{subspace} function in \texttt{MATLAB}.

The experiments are conducted on a workstation named \textit{nightcrawler} provided by the Mathematical Institute, University of Oxford, equipped with a 64-core (128-thread) processor running at a base frequency of 2.50~GHz (up to 4.10~GHz boost), 2048~GB of system memory, and 9.7~TB of storage, running a Linux operating system. The code is available at \url{https://github.com/fhrobat/SubApSnapEig}.

\pgfplotstableset{
    times table style/.style={
        col sep=tab,
        every head row/.style={before row=\toprule,after row=\midrule},
        every last row/.style={after row=\bottomrule},
        columns={Metodo, Offline_sec, Online_sec, Totale_sec, Ratio_Off_vs_POD, Ratio_On_vs_POD, Ratio_vs_POD},
        columns/Metodo/.style={
        string type, 
        column name=Method,
        string replace={POD}{MOR},
        postproc cell content/.append style={
        @cell content/.add={\ttfamily}{}}},
        columns/Offline_sec/.style={column name=Offline(s), numeric type, fixed, fixed zerofill, precision=3},
        columns/Online_sec/.style={column name=Online(s), numeric type, fixed, fixed zerofill, precision=3},
        columns/Totale_sec/.style={column name=Total(s), numeric type, fixed, fixed zerofill, precision=3},
        columns/Ratio_Off_vs_POD/.style={column name=Ratio Off., numeric type, fixed, fixed zerofill, precision=3},
        columns/Ratio_On_vs_POD/.style={column name=Ratio On., numeric type, fixed, fixed zerofill, precision=3, postproc cell content/.append style={
        @cell content/.add={\bfseries\boldmath}{}
        }},
        columns/Ratio_vs_POD/.style={column name=Ratio Tot., numeric type, fixed, fixed zerofill, precision=3},
    }
}

\subsection{Example 1: Exponential Differential Operator}\label{sub:exp}

In the first example, we compute the smallest eigenvalue of the 3-point stencil finite-difference discretization of the following differential operator
\[ \frac{d}{dx} \left(e^{p x}\frac{du}{dx} \right) \qquad x \in \Omega =[-10,10]\]
with zero boundary conditions. This is a particular example of a Sturm-Liouville eigenvalue problem \cite[\S 2.4]{Al-2008}.
The parameter space is the interval $p\in[0.1,1]$, 
which is discretized using $3000$ log-spaced points, while we have discretized the domain $\Omega$ with a uniform grid such that the resulting dimension of $A(p) \in \R^{n \times n}$ is $n = 2\times 10^6$. We divide each matrix by $n^2$ to have $\|A(p)\|_2 \approx 1$.

In \Cref{fig:exp_k5_tol1e-6_freqInf_equi} and \Cref{tab:exp_k5_tol1e-6_freqInf_equi} we report the results for $10$ log-spaced snapshot points in the parameter vector.

 The error on the eigenvalue and eigenvector for \texttt{SubApSnapEig} is of order $10^{-4}$ while the \texttt{MOR} method achieves an error of approximately $4$ orders of magnitude smaller. We also note that the error does not significantly dip around the snapshot points.

The \texttt{SubApSnapEig} method is $20 \times$ faster than \texttt{MOR} method and the \texttt{LU} variant is faster than the \texttt{Leverage} variant. This is because we have to compute more rows of $A(p)$ in the \texttt{Leverage} variant.

% ==================== EXP FIG 1 (EQUI) ====================
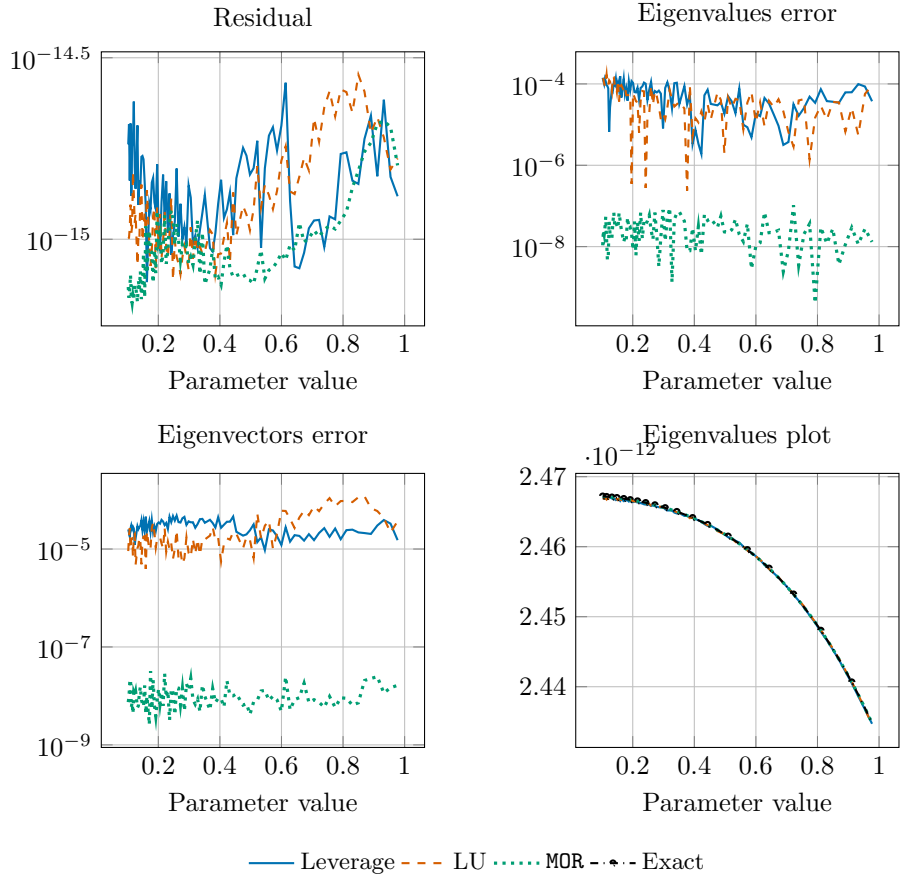
\begin{figure}[ht]
\centering
\begin{tikzpicture}
\begin{groupplot}[
    group style={group size=2 by 2, horizontal sep=2cm, vertical sep=0.15\textwidth},
    width=0.45\textwidth,
    height=0.4\textwidth,
    grid=both,
    legend style={
    at={(-0.3,-0.5)},
    anchor=south,
    legend columns=4,
    font=\small
    },
]

% Residual
\nextgroupplot[title={Residual}, xlabel={Parameter value}, ymode=log]
\addplot+ [mark=none] table [x=p, y expr=\thisrow{Res_Leverage}/(4e12)] {experiments_new/exponential/k10_s0_tolInf_freqInf/res.dat};
\addplot+ [mark=none] table [x=p, y expr=\thisrow{Res_LU}/(4e12)] {experiments_new/exponential/k10_s0_tolInf_freqInf/res.dat};
\addplot+ [mark=none] table [x=p, y expr=\thisrow{Res_POD}/(4e12)] {experiments_new/exponential/k10_s0_tolInf_freqInf/res.dat};

% Eigenvalues
\nextgroupplot[title={Eigenvalues error}, xlabel={Parameter value}, ymode=log]
\addplot+ [mark=none] table [x=p, y=ErrLam_Leverage] {experiments_new/exponential/k10_s0_tolInf_freqInf/eigval.dat};
\addplot+ [mark=none] table [x=p, y=ErrLam_LU] {experiments_new/exponential/k10_s0_tolInf_freqInf/eigval.dat};
\addplot+ [mark=none] table [x=p, y=ErrLam_POD] {experiments_new/exponential/k10_s0_tolInf_freqInf/eigval.dat};

% Eigenvectors
\nextgroupplot[title={Eigenvectors error}, xlabel={Parameter value}, ymode=log]
\addplot+ [mark=none] table [x=p, y=ErrAng_Leverage] {experiments_new/exponential/k10_s0_tolInf_freqInf/eigvec.dat};
\addplot+ [mark=none] table [x=p, y=ErrAng_LU] {experiments_new/exponential/k10_s0_tolInf_freqInf/eigvec.dat};
\addplot+ [mark=none] table [x=p, y=ErrAng_POD] {experiments_new/exponential/k10_s0_tolInf_freqInf/eigvec.dat};

% Plot
\nextgroupplot[title={Eigenvalues plot}, xlabel={Parameter value}]
\addplot+ [mark=none] table [x=p, y expr=\thisrow{Lev_eig_1}/(4e12)] {experiments_new/exponential/k10_s0_tolInf_freqInf/plot.dat};
\addlegendentry{Leverage}

\addplot+ [mark=none] table [x=p, y expr=\thisrow{LU_eig_1}/(4e12)] {experiments_new/exponential/k10_s0_tolInf_freqInf/plot.dat};
\addlegendentry{LU}

\addplot+ [mark=none] table [x=p, y expr=\thisrow{POD_eig_1}/(4e12)] {experiments_new/exponential/k10_s0_tolInf_freqInf/plot.dat};
\addlegendentry{\texttt{MOR}}

\addplot+ [mark=o, mark repeat=5] table [x=p, y expr=\thisrow{Ex_eig_1}/(4e12)] {experiments_new/exponential/k10_s0_tolInf_freqInf/plot.dat};
\addlegendentry{Exact}

\end{groupplot}
\end{tikzpicture}
\caption{Results for experiments Exponential Differential Operator in \Cref{sub:exp} for $10$ log-spaced snapshot points.
}\label{fig:exp_k5_tol1e-6_freqInf_equi}
\end{figure}

\begin{table}[ht]
\centering
\caption{Computational times for experiments Exponential Differential Operator in \Cref{sub:exp} for $10$ log-spaced snapshot points.}\label{tab:exp_k5_tol1e-6_freqInf_equi}
\begin{filecontents*}{experiments_new/exponential/k10_s0_tolInf_freqInf/times.dat}
Metodo	Offline_sec	Online_sec	Totale_sec	Ratio_Off_vs_POD	Ratio_On_vs_POD	Ratio_vs_POD
LU	9.281922	64.555181	73.837103	0.999341518628584	0.0468457673647863	0.0532226575490803
Leverage	9.22374	71.112848	80.336588	0.993077332370948	0.0516044705080358	0.0579075632448033
POD	9.288038	1378.036579	1387.324617	1	1	1
\end{filecontents*}
\pgfplotstabletypeset[times table style]{experiments_new/exponential/k10_s0_tolInf_freqInf/times.dat}
\end{table}

\subsection{Example 2: Differential operator with multiple eigenvalues}\label{sub:ns}

In the second example, we compute the smallest $3$ eigenvalues of the $5$-point stencil finite-difference discretization of the differential operator
\[-p \Delta \mathbf{u} - \nabla (\nabla \cdot \mathbf{u}) \quad \Omega = [-1,1]^2\]
with zero boundary conditions. The computation of the smallest eigenvalue of this operator is studied in various works \cite{Hen2026, Kha2026}.
Note that the discretization of the operator has the second smallest eigenvalue of multiplicity $2$ because of the symmetry of the system. For this experiment,  we choose the dimension of the matrix $n$ equal to $250000$.  The parameter $p$ is log-spaced between $10^{-4}$ and $10^{-3}$. 

In \Cref{fig:ns_k10_s0_tolInf_freqInf} and \Cref{tab:ns_k10_s0_tolInf_freqInf} we report the results for $10$ log-spaced snapshot points. 
Again we divide the matrices by $n^2$ to get $\|A(p)\|_2 \approx 1$. We observe that the error for \texttt{SubApSnapEig} methods is two orders of
 magnitude larger than \texttt{MOR}. In this example, we observe that the errors dip around the snapshot points.

In this case \texttt{LU} variant of \texttt{SubApSnapEig} is significantly faster, as a single row of $A(p)$ has more non-zero entries than the example in \Cref{sub:exp}. \texttt{SubApSnapEig} is approximately $25\times$ faster.

% ==================== NS FIG 1 ====================
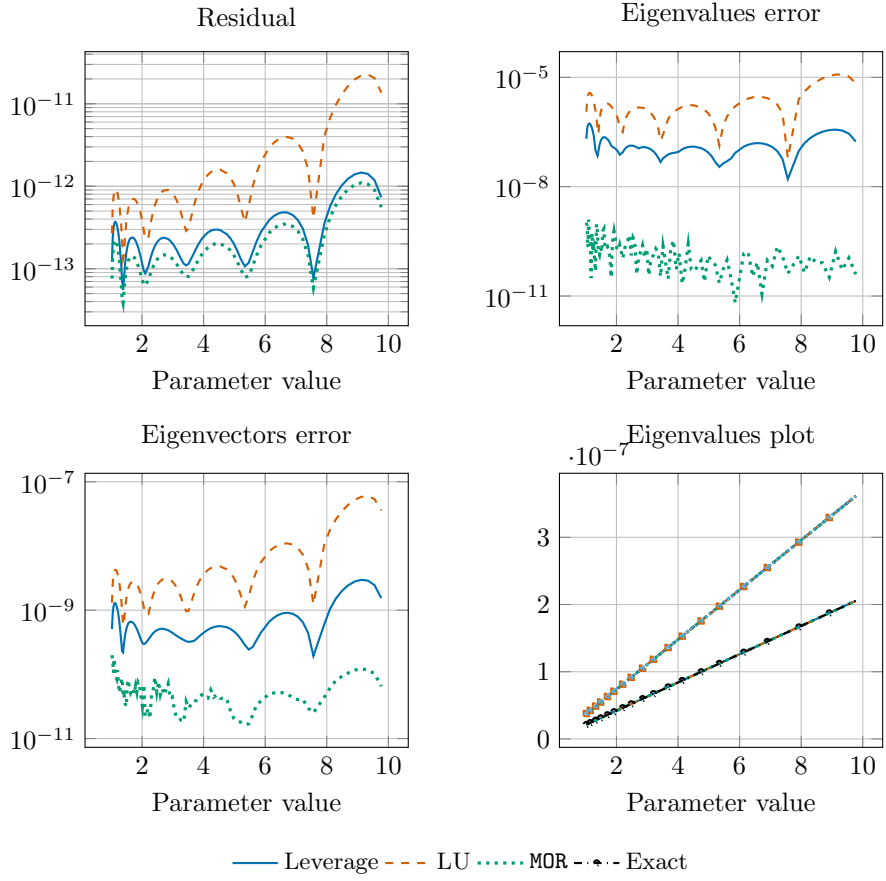
\begin{figure}[ht]
\centering
\caption{Results for experiments for the differential operator in \Cref{sub:ns} for $10$ log-spaced snapshot points.}\label{fig:ns_k10_s0_tolInf_freqInf}
\begin{tikzpicture}
\begin{groupplot}[
    group style={group size=2 by 2, horizontal sep=2cm, vertical sep=0.15\textwidth},
    width=0.45\textwidth,
    height=0.4\textwidth,
    grid=both,
    legend style={
    at={(-0.3,-0.5)},
    anchor=south,
    legend columns=4,
    font=\small
    },
]

% Residual
\nextgroupplot[title={Residual}, xlabel={Parameter value}, ymode=log]
\addplot+ [mark=none] table [x=p, y expr=\thisrow{Res_Leverage}/(sqrt(3)*6.25e10)] {experiments_new/nc/k10_s0_tolInf_freqInf/res.dat};
\addplot+ [mark=none] table [x=p, y expr=\thisrow{Res_LU}/(sqrt(3)*6.25e10)] {experiments_new/nc/k10_s0_tolInf_freqInf/res.dat};
\addplot+ [mark=none] table [x=p, y expr=\thisrow{Res_POD}/(sqrt(3)*6.25e10)] {experiments_new/nc/k10_s0_tolInf_freqInf/res.dat};

% Eigenvalues
\nextgroupplot[title={Eigenvalues error}, xlabel={Parameter value}, ymode=log]
\addplot+ [mark=none] table [x=p, y=ErrLam_Leverage] {experiments_new/nc/k10_s0_tolInf_freqInf/eigval.dat};
\addplot+ [mark=none] table [x=p, y=ErrLam_LU] {experiments_new/nc/k10_s0_tolInf_freqInf/eigval.dat};
\addplot+ [mark=none] table [x=p, y=ErrLam_POD] {experiments_new/nc/k10_s0_tolInf_freqInf/eigval.dat};

% Eigenvectors
\nextgroupplot[title={Eigenvectors error}, xlabel={Parameter value}, ymode=log]
\addplot+ [mark=none] table [x=p, y=ErrAng_Leverage] {experiments_new/nc/k10_s0_tolInf_freqInf/eigvec.dat};
\addplot+ [mark=none] table [x=p, y=ErrAng_LU] {experiments_new/nc/k10_s0_tolInf_freqInf/eigvec.dat};
\addplot+ [mark=none] table [x=p, y=ErrAng_POD] {experiments_new/nc/k10_s0_tolInf_freqInf/eigvec.dat};

% Plot
\nextgroupplot[title={Eigenvalues plot}, xlabel={Parameter value}]

% --- PRIMI AUTOVALORI (Linea continua) ---
\addplot+ [mark=none] table [x=p, y expr=\thisrow{Lev_eig_1}/(6.25e10)] {experiments_new/nc/k10_s0_tolInf_freqInf/plot.dat};
\addlegendentry{Leverage}
\addplot+ [mark=none] table [x=p, y expr=\thisrow{LU_eig_1}/(6.25e10)] {experiments_new/nc/k10_s0_tolInf_freqInf/plot.dat};
\addlegendentry{LU}
\addplot+ [mark=none] table [x=p, y expr=\thisrow{POD_eig_1}/(6.25e10)] {experiments_new/nc/k10_s0_tolInf_freqInf/plot.dat};
\addlegendentry{\texttt{MOR}}
\addplot+ [mark=o, mark repeat=5] table [x=p, y expr=\thisrow{Ex_eig_1}/(6.25e10)] {experiments_new/nc/k10_s0_tolInf_freqInf/plot.dat};
\addlegendentry{Exact}

% --- SECONDI AUTOVALORI (Linea tratteggiata) ---
\addplot+ [mark=none, densely dashed] table [x=p, y expr=\thisrow{Lev_eig_2}/(6.25e10)] {experiments_new/nc/k10_s0_tolInf_freqInf/plot.dat};
\addplot+ [mark=none, densely dashed] table [x=p, y expr=\thisrow{LU_eig_2}/(6.25e10)] {experiments_new/nc/k10_s0_tolInf_freqInf/plot.dat};
\addplot+ [mark=none, densely dashed] table [x=p, y expr=\thisrow{POD_eig_2}/(6.25e10)] {experiments_new/nc/k10_s0_tolInf_freqInf/plot.dat};
\addplot+ [mark=square, mark repeat=5, densely dashed] table [x=p, y expr=\thisrow{Ex_eig_2}/(6.25e10)] {experiments_new/nc/k10_s0_tolInf_freqInf/plot.dat};

% --- TERZI AUTOVALORI (Linea puntinata) ---
\addplot+ [mark=none, densely dotted] table [x=p, y expr=\thisrow{Lev_eig_3}/(6.25e10)] {experiments_new/nc/k10_s0_tolInf_freqInf/plot.dat};
\addplot+ [mark=none, densely dotted] table [x=p, y expr=\thisrow{LU_eig_3}/(6.25e10)] {experiments_new/nc/k10_s0_tolInf_freqInf/plot.dat};
\addplot+ [mark=none, densely dotted] table [x=p, y expr=\thisrow{POD_eig_3}/(6.25e10)] {experiments_new/nc/k10_s0_tolInf_freqInf/plot.dat};
\addplot+ [mark=triangle, mark repeat=5, densely dotted] table [x=p, y expr=\thisrow{Ex_eig_3}/(6.25e10)] {experiments_new/nc/k10_s0_tolInf_freqInf/plot.dat};

\end{groupplot}
\end{tikzpicture}
\end{figure}

\begin{table}[ht]
\centering
\caption{Computational times for experiments for the differential operator in \Cref{sub:ns} for $10$ log-spaced snapshot points. }\label{tab:ns_k10_s0_tolInf_freqInf}
\begin{filecontents*}{experiments_new/nc/k10_s0_tolInf_freqInf/times.dat}
Metodo	Offline_sec	Online_sec	Totale_sec	Ratio_Off_vs_POD	Ratio_On_vs_POD	Ratio_vs_POD
LU	249.480027	11.511795	260.991822	0.976601449167882	0.0249348179104987	0.363937883768313
Leverage	254.718087	26.565296	281.283383	0.997106084542274	0.05754105406659	0.392233282881224
POD	255.457359	461.675519	717.132878	1	1	1
\end{filecontents*}
\pgfplotstabletypeset[times table style]{experiments_new/nc/k10_s0_tolInf_freqInf/times.dat}
\end{table}

\subsection{Example 3: PageRank}\label{sub:pagerank}

In this experiment, we compute the leading eigenvalue and eigenvector of a parametric PageRank problem \cite{Bri1998} from the \texttt{web-NotreDame} web graph dataset of dimension $n=325729$. We first remove existing self-loops while we add self-loops on dangling nodes. The base column-stochastic transition matrix is then constructed as $M = A D^{-1}$, where $D$ is the diagonal matrix containing the out-degrees of the nodes. 

We consider a degree-biased PageRank model where the damping factor is fixed at $d = 0.9$, while the link weights are modulated dynamically by a parameter $p$. Dynamic PageRank has been investigated in various works, see \cite{Ban2022, Mei2010}. Specifically, the matrix $A(p)$ is first obtained computing a column-stochastic matrix $W(p)$ whose entries are obtained by column-normalizing the weighted matrix
\[  M_{ij}(1 + \text{in\_degree}_i)^p\]
and adding a rank-one matrix
\[ A(p) = dW(p) + \frac{(1-d)}{n} \vec{1}\vec{1}^T. \]
 The vector $\vec{1}$ is the vector of all ones. The parameter space  is the interval $p\in[0, 0.5]$, which is discretized using $1000$ equispaced points.

Even though $A(p)$ is not affine, we can compute matvecs with $A(p)$ by computing matvecs with $W(p)$, which is sparse and with $\vec{1}\vec{1}^T$. As the matrix $A(p)$ is dense, we modify our implementations to handle only matvecs with $A(p)$ or to extract some rows from $A(p)$. This makes the online phase of \texttt{MOR} method feasible and fast (but not faster than \texttt{SubApSnapEig}).

In \Cref{fig:pagerank_k10_tol1e-6_freqInf_equi} and \Cref{tab:pagerank_k10_tol1e-6_freqInf_equi} we report the performances of  using $5$ equispaced snapshot points across the parameter domain.

In this example the errors of all our methods are almost the same and \texttt{MOR} is more accurate except for the eigenvalue error, where \texttt{LU} variant of \texttt{SubApSnapEig} achieves an accuracy of $2$ orders of magnitude smaller. This behaviour is unclear.

As said before, the \texttt{MOR} method is only $2 \times$ slower for this example thanks to the efficient matvec implementation.

% ==================== PAGERANK FIG 1 (EQUI) ====================
\begin{figure}[ht]
\centering
\begin{tikzpicture}
\begin{groupplot}[
    group style={group size=2 by 2, horizontal sep=2cm, vertical sep=0.15\textwidth},
    width=0.45\textwidth,
    height=0.4\textwidth,
    grid=both,
    legend style={
    at={(-0.3,-0.5)},
    anchor=south,
    legend columns=4,
    font=\small
    },
]

% Residual
\nextgroupplot[title={Residual}, xlabel={Parameter value}, ymode=log]
\addplot+ [mark=none] table [x=p, y=Res_Leverage] {experiments_new/pagerank/k5_s0_tolInf_freqInf/res.dat};
\addplot+ [mark=none] table [x=p, y=Res_LU] {experiments_new/pagerank/k5_s0_tolInf_freqInf/res.dat};
\addplot+ [mark=none] table [x=p, y=Res_POD] {experiments_new/pagerank/k5_s0_tolInf_freqInf/res.dat};

% Eigenvalues
\nextgroupplot[title={Eigenvalues error}, xlabel={Parameter value}, ymode=log]
\addplot+ [mark=none] table [x=p, y=ErrLam_Leverage] {experiments_new/pagerank/k5_s0_tolInf_freqInf/eigval.dat};
\addplot+ [mark=none] table [x=p, y=ErrLam_LU] {experiments_new/pagerank/k5_s0_tolInf_freqInf/eigval.dat};
\addplot+ [mark=none] table [x=p, y=ErrLam_POD] {experiments_new/pagerank/k5_s0_tolInf_freqInf/eigval.dat};

% Eigenvectors
\nextgroupplot[title={Eigenvectors error}, xlabel={Parameter value}, ymode=log]
\addplot+ [mark=none] table [x=p, y=ErrAng_Leverage] {experiments_new/pagerank/k5_s0_tolInf_freqInf/eigvec.dat};
\addplot+ [mark=none] table [x=p, y=ErrAng_LU] {experiments_new/pagerank/k5_s0_tolInf_freqInf/eigvec.dat};
\addplot+ [mark=none] table [x=p, y=ErrAng_POD] {experiments_new/pagerank/k5_s0_tolInf_freqInf/eigvec.dat};

% Plot
\nextgroupplot[
title={Eigenvalues plot}, 
xlabel={Parameter value},
enlarge y limits=0.1,
yticklabel style={
        /pgf/number format/.cd,
        fixed,
        fixed zerofill,
        precision=6
    }
]
\addplot+ [mark=none] table [x=p, y=Lev_eig_1] {experiments_new/pagerank/k5_s0_tolInf_freqInf/plot.dat}; \addlegendentry{Leverage}
\addplot+ [mark=none] table [x=p, y=LU_eig_1] {experiments_new/pagerank/k5_s0_tolInf_freqInf/plot.dat}; \addlegendentry{LU}
\addplot+ [mark=none] table [x=p, y=POD_eig_1] {experiments_new/pagerank/k5_s0_tolInf_freqInf/plot.dat}; \addlegendentry{\texttt{MOR}}
\addplot+ [mark = o, mark repeat=10] table [x=p, y=Ex_eig_1] {experiments_new/pagerank/k5_s0_tolInf_freqInf/plot.dat}; \addlegendentry{Exact}

\end{groupplot}
\end{tikzpicture}
\caption{Results for experiment PageRank in \Cref{sub:pagerank} for $5$ equispaced snapshot points.}\label{fig:pagerank_k10_tol1e-6_freqInf_equi}
\end{figure}
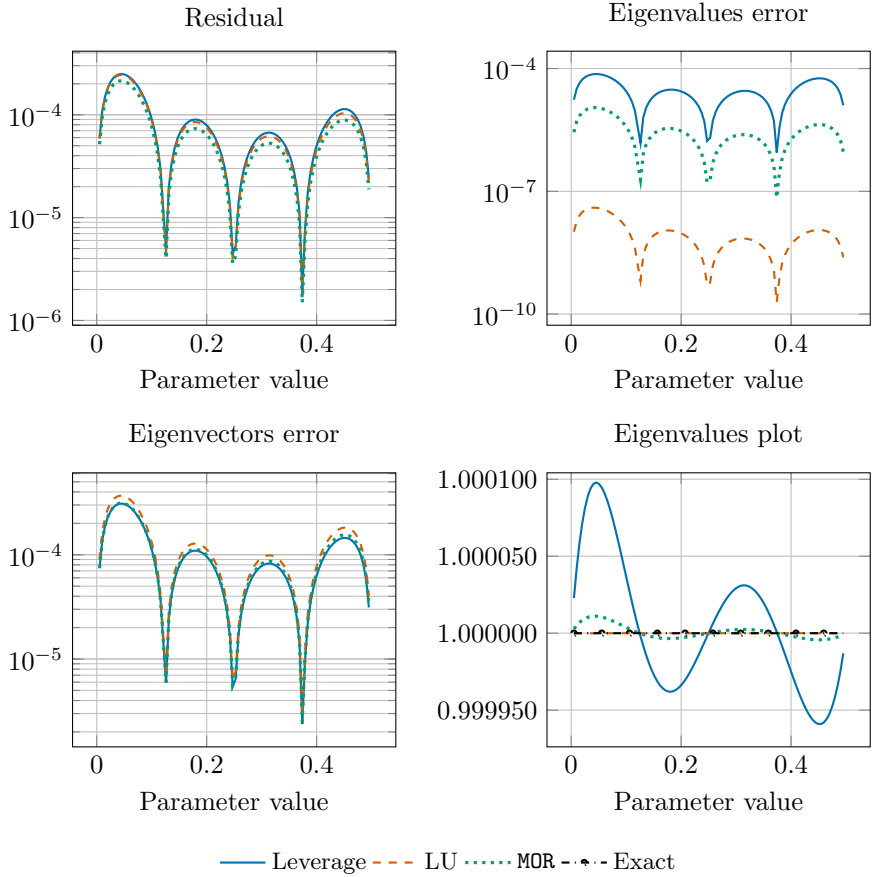

\begin{table}[ht]
\centering
\caption{Computational times for experiment PageRank in \Cref{sub:pagerank} for $5$ equispaced snapshot points.}\label{tab:pagerank_k10_tol1e-6_freqInf_equi}
\begin{filecontents*}{experiments_new/pagerank/k5_s0_tolInf_freqInf/times.dat}
Metodo	Offline_sec	Online_sec	Totale_sec	Ratio_Off_vs_POD	Ratio_On_vs_POD	Ratio_vs_POD
LU	12.157492	11.740511	23.898003	1.02983119905636	0.453197015236948	0.633709322306329
Leverage	12.76067	12.064548	24.825218	1.08092492159259	0.465705210257278	0.658296514352554
POD	11.805325	25.905976	37.711301	1	1	1
\end{filecontents*}
\pgfplotstabletypeset[
    times table style,
    display columns/5/.style={fixed, precision=4}, 
]{experiments_new/pagerank/k5_s0_tolInf_freqInf/times.dat}
\end{table}

\subsection{Example 4: Gaussian Kernel}\label{sub:kern}
In this example, we apply our algorithms to a different setting than discretization of PDEs. In particular, we apply our algorithm to the Gaussian kernel matrix \cite{Ras2005}
produced from the ``regular" MNIST dataset \cite{Lec1998}. The dimension of the dataset is $n = 60000$. The gaussian kernel  can be seen as a pointwise exponential of the matrix where each entry represents the square distance between two points in the dataset divided by $2p^2$ where $p$ is a parameter.
The spectral properties of the gaussian kernel are crucial in statistics and machine learning \cite{Rea1983, Ha1986}. The matrix $A(p)$ is defined as  
\[ A_{ij}(p)=\exp(-\|x_i-x_j\|^2/2p^2).\]
The set of parameters
is chosen as $1000$ equispaced points in $[0.5p_*, 1.5p_*]$ where $p_*$ is the median of the distances between two different objects in the dataset.

In \Cref{fig:kernel_k5_tol1e-8_freqInf} and \Cref{tab:kernel_k5_tol1e-8_freqInf} we report the results for $10$ equispaced snapshot points.
In this example the errors of the different algorithms are pretty similar except the \texttt{MOR} error on eigenvalues for the first half of the parameter domain, which is order of machine precision and several order of magnitudes better than \texttt{SubApSnapEig} error.

\texttt{SubApSnapEig} is $250 \times$ faster in the online phase as we have to compute only $sn$ scalar exponentials instead of $n^2$.

% ==================== KERNEL FIG 1 ====================
\begin{figure}[ht]
\centering
\begin{tikzpicture}
\begin{groupplot}[
    group style={group size=2 by 2, horizontal sep=2cm, vertical sep=0.15\textwidth},
    width=0.45\textwidth,
    height=0.4\textwidth,
    grid=both,
    legend style={
    at={(-0.3,-0.5)},
    anchor=south,
    legend columns=4,
    font=\small
    },
]

% Residual
\nextgroupplot[title={Residual}, xlabel={Parameter value}, ymode=log]
\addplot+ [mark=none] table [x=p, y=Res_Leverage] {experiments_new/kernel/k10_s0_tolInf_freqInf/res.dat};
\addplot+ [mark=none] table [x=p, y=Res_LU] {experiments_new/kernel/k10_s0_tolInf_freqInf/res.dat};
\addplot+ [mark=none] table [x=p, y=Res_POD] {experiments_new/kernel/k10_s0_tolInf_freqInf/res.dat};

% Eigenvalues
\nextgroupplot[title={Eigenvalues error}, xlabel={Parameter value}, ymode=log]
\addplot+ [mark=none] table [x=p, y=ErrLam_Leverage] {experiments_new/kernel/k10_s0_tolInf_freqInf/eigval.dat};
\addplot+ [mark=none] table [x=p, y=ErrLam_LU] {experiments_new/kernel/k10_s0_tolInf_freqInf/eigval.dat};
\addplot+ [mark=none] table [x=p, y=ErrLam_POD] {experiments_new/kernel/k10_s0_tolInf_freqInf/eigval.dat};

% Eigenvectors
\nextgroupplot[title={Eigenvectors error}, xlabel={Parameter value}, ymode=log]
\addplot+ [mark=none] table [x=p, y=ErrAng_Leverage] {experiments_new/kernel/k10_s0_tolInf_freqInf/eigvec.dat};
\addplot+ [mark=none] table [x=p, y=ErrAng_LU] {experiments_new/kernel/k10_s0_tolInf_freqInf/eigvec.dat};
\addplot+ [mark=none] table [x=p, y=ErrAng_POD] {experiments_new/kernel/k10_s0_tolInf_freqInf/eigvec.dat};

% Plot
\nextgroupplot[title={Eigenvalues plot}, xlabel={Parameter value}]
\addplot+ [mark=none] table [x=p, y=Lev_eig_1] {experiments_new/kernel/k10_s0_tolInf_freqInf/plot.dat}; \addlegendentry{Leverage}
\addplot+ [mark=none] table [x=p, y=LU_eig_1] {experiments_new/kernel/k10_s0_tolInf_freqInf/plot.dat}; \addlegendentry{LU}
\addplot+ [mark=none] table [x=p, y=POD_eig_1] {experiments_new/kernel/k10_s0_tolInf_freqInf/plot.dat}; \addlegendentry{\texttt{MOR}}
\addplot+ [mark = o, mark repeat=5] table [x=p, y=Ex_eig_1] {experiments_new/kernel/k10_s0_tolInf_freqInf/plot.dat}; \addlegendentry{Exact}

\end{groupplot}
\end{tikzpicture}
\caption{Results for experiments Gaussian Kernel in \Cref{sub:kern} for $10$ equispaced snapshot points. No periodic updates.}\label{fig:kernel_k5_tol1e-8_freqInf}
\end{figure}
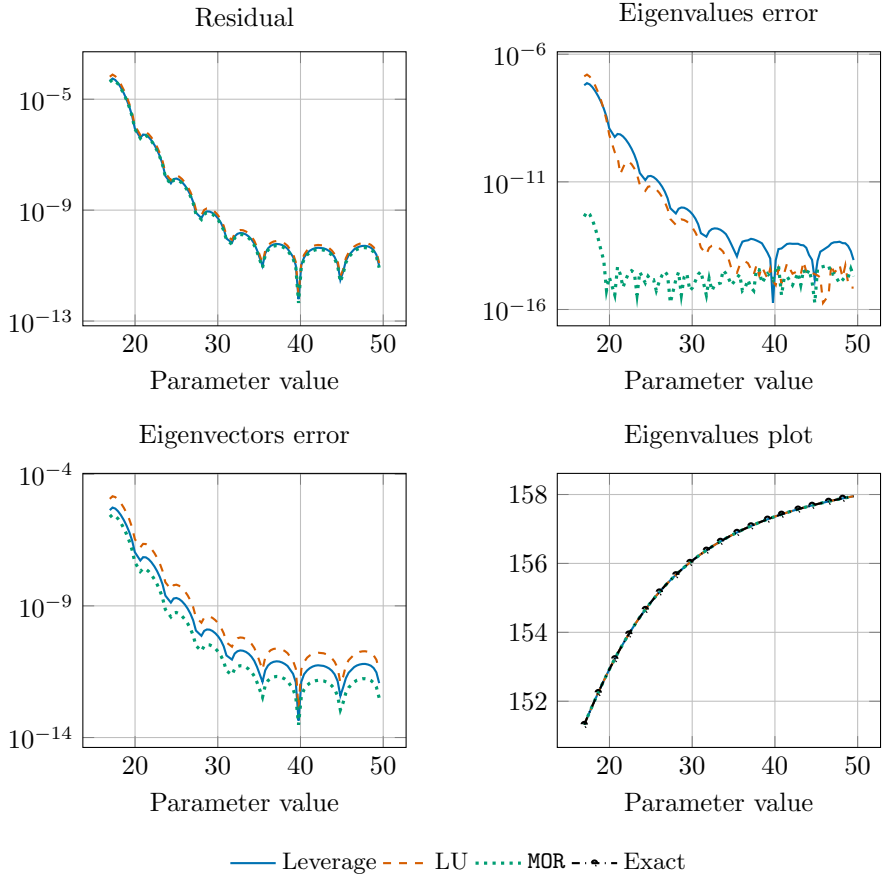

\begin{table}[ht]
\centering
\caption{Computational times for experiments Gaussian Kernel in \Cref{sub:kern} for $10$ equispaced snapshot points.}\label{tab:kernel_k5_tol1e-8_freqInf}
\begin{filecontents*}{experiments_new/kernel/k10_s0_tolInf_freqInf/times.dat}
Metodo	Offline_sec	Online_sec	Totale_sec	Ratio_Off_vs_POD	Ratio_On_vs_POD	Ratio_vs_POD
LU	13.033052	1.882018	14.91507	1.43689046369603	0.003887277981233	0.0302402980531424
Leverage	13.32906	1.992964	15.322024	1.46952526576524	0.00411643516405796	0.0310653971142878
POD	9.070317	484.148036	493.218353	1	1	1
\end{filecontents*}
\pgfplotstabletypeset[times table style]{experiments_new/kernel/k10_s0_tolInf_freqInf/times.dat}
\end{table}

\subsection{Example 5: Graph Eigenvector Centrality}\label{sub:eigcen}

In this example, we compute the leading eigenpair to evaluate the \textit{eigenvector centrality} \cite[\S 7]{New2010} of nodes within a parametrically varying network. In particular, the parameter-dependent matrix is defined as a symmetric diagonal scaling, taking the form $A(p) = D(p) A_{\text{0}} D(p)$. Here, $A_{\text{0}}$ is the adjacency matrix of the graph \textit{coAuthors-DBLP} \cite{Dav2011, Bad2013} with dimension $n = 299067$. The parameter dependence is embedded in the diagonal matrix $D(p)$, where each $i$-th entry is equal to $1 + 0.5 \sin(2\pi p / 1000 + \Phi_i)$, with the phases $\Phi_i$ being uniformly distributed random variables in $[0, 2\pi]$. The set of parameters is chosen as $1000$ integer points uniformly spaced in the interval $[1, 1000]$. This model simulates the evolution of the intensity of connections between nodes.

In \Cref{fig:graph_k50_freq20_start} and \Cref{tab:graph_k50_freq20_start} we report the results for $20$ snapshot points concentrated at the start of the parameter space. We force an update of the snapshot matrix every $100$ steps.

In \Cref{fig:graph_k50_freq20_start}  we can clearly see when the snapshot matrix is updated and how the errors are reduced for the upcoming value of the parameters. Furthermore, not updating the snapshot matrix achieves a worse accuracy. Towards the end of the parameter domain, as the matrix $A(p)$ is periodic, the errors decrease as expected. 

The \texttt{SubApSnapEig} method is $\approx6\times$ faster in the online phase. Note that, the more frequently the snapshot matrix is updated, the less efficient is our algorithm compared to the \texttt{MOR} method.

% In \Cref{fig:graph_k50_freq100_start} and \Cref{tab:graph_k50_freq100_start} we report the results for $50$ snapshot points concentrated at the start of the parameter space. We force an update of the snapshot matrix every $100$ steps.

% ==================== GRAPH FIG 1 ====================
\begin{figure}[ht]
\centering
\begin{tikzpicture}
\begin{groupplot}[
    group style={group size=2 by 2, horizontal sep=2cm, vertical sep=0.15\textwidth},
    width=0.45\textwidth,
    height=0.4\textwidth,
    grid=both,
    legend style={
    at={(-0.3,-0.5)},
    anchor=south,
    legend columns=4,
    font=\small
    },
]

% Residual
\nextgroupplot[title={Residual}, xlabel={Parameter value}, ymode=log]
\addplot+ [mark=none] table [x=p, y=Res_Leverage] {experiments_new/eig_centrality/k20_s0_tolInf_freq100/res.dat};
\addplot+ [mark=none] table [x=p, y=Res_LU] {experiments_new/eig_centrality/k20_s0_tolInf_freq100/res.dat};
\addplot+ [mark=none] table [x=p, y=Res_POD] {experiments_new/eig_centrality/k20_s0_tolInf_freq100/res.dat};

% Eigenvalues
\nextgroupplot[title={Eigenvalues error}, xlabel={Parameter value}, ymode=log]
\addplot+ [mark=none] table [x=p, y=ErrLam_Leverage] {experiments_new/eig_centrality/k20_s0_tolInf_freq100/eigval.dat};
\addplot+ [mark=none] table [x=p, y=ErrLam_LU] {experiments_new/eig_centrality/k20_s0_tolInf_freq100/eigval.dat};
\addplot+ [mark=none] table [x=p, y=ErrLam_POD] {experiments_new/eig_centrality/k20_s0_tolInf_freq100/eigval.dat};

% Eigenvectors
\nextgroupplot[title={Eigenvectors error}, xlabel={Parameter value}, ymode=log]
\addplot+ [mark=none] table [x=p, y=ErrAng_Leverage] {experiments_new/eig_centrality/k20_s0_tolInf_freq100/eigvec.dat};
\addplot+ [mark=none] table [x=p, y=ErrAng_LU] {experiments_new/eig_centrality/k20_s0_tolInf_freq100/eigvec.dat};
\addplot+ [mark=none] table [x=p, y=ErrAng_POD] {experiments_new/eig_centrality/k20_s0_tolInf_freq100/eigvec.dat};

% Plot
\nextgroupplot[title={Eigenvalues plot}, xlabel={Parameter value}]
\addplot+ [mark=none] table [x=p, y=Lev_eig_1] {experiments_new/eig_centrality/k20_s0_tolInf_freq100/plot.dat}; \addlegendentry{Leverage}
\addplot+ [mark=none] table [x=p, y=LU_eig_1] {experiments_new/eig_centrality/k20_s0_tolInf_freq100/plot.dat}; \addlegendentry{LU}
\addplot+ [mark=none] table [x=p, y=POD_eig_1] {experiments_new/eig_centrality/k20_s0_tolInf_freq100/plot.dat}; \addlegendentry{\texttt{MOR}}
\addplot+ [mark = o, mark repeat=5] table [x=p, y=Ex_eig_1] {experiments_new/eig_centrality/k20_s0_tolInf_freq100/plot.dat}; \addlegendentry{Exact}

\end{groupplot}
\end{tikzpicture}
\caption{Results for experiments Graph Eigenvector Centrality in \Cref{sub:eigcen} for $20$ snapshot points at the beginning. Snapshot matrix updated every $100$ steps.}\label{fig:graph_k50_freq20_start}
\end{figure}
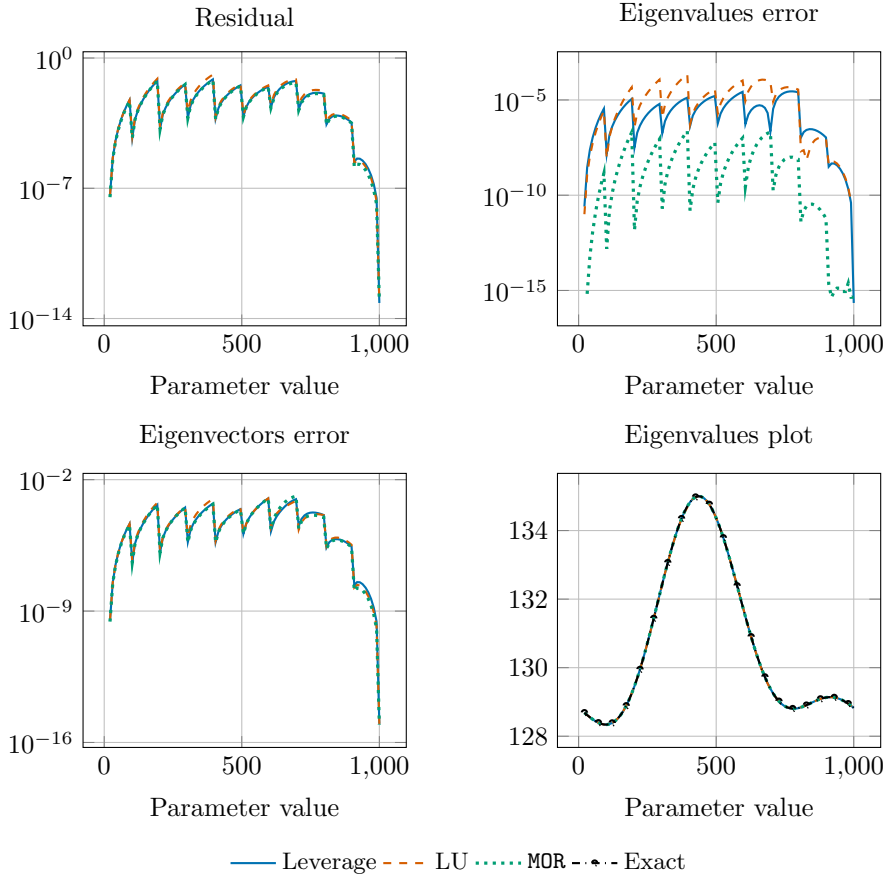

\begin{table}[ht]
\centering
\caption{Computational times for experiments Graph Eigenvector Centrality in \Cref{sub:eigcen} for $20$ snapshot points at the beginning. Snapshot matrix updated every $100$ steps.}\label{tab:graph_k50_freq20_start}
\begin{filecontents*}{experiments_new/eig_centrality/k20_s0_tolInf_freq100/times.dat}
Metodo	Offline_sec	Online_sec	Totale_sec	Ratio_Off_vs_POD	Ratio_On_vs_POD	Ratio_vs_POD
LU	7.213743	10.132686	17.346429	1.07929833832406	0.175990834089708	0.269946307130608
Leverage	6.800076	10.393932	17.194008	1.01740673701257	0.180528318172665	0.267574321168589
POD	6.683734	57.575078	64.258812	1	1	1
\end{filecontents*}
\pgfplotstabletypeset[times table style]{experiments_new/eig_centrality/k20_s0_tolInf_freq100/times.dat}
\end{table}

\subsection{Example 6: Transformer}\label{sub:trans}

As a final example, we compute the leading eigenvalue of the attention matrix of a Transformer-based architecture during the fine-tuning process. See \cite{Vas2017} for the definitions of the matrices we will use later. 

We consider the pre-trained \textit{ModernBERT-base} model \cite{War2024} (model name\\ ``answerdotai/ModernBERT-base'' in the \texttt{transformers} library in Python) and we fine-tune it on the SCOTUS subset of the LexGLUE dataset \cite{Cha2022} (dataset name ``lex\_glue'' in the \texttt{datasets} library in Python) for $4$ epochs. The network is fine-tuned using the AdamW optimizer with a batch size of $8$ and a peak learning rate of $3 \times 10^{-5}$, governed by a cosine learning rate scheduler with a $10\%$ warmup ratio.

Before training, we select a specific document of length $5000$ tokens which will remain fixed throughout the entire process.

Denote with $p$ the number of the step of the training of the model in the last epoch.
The attention mechanism of this architecture includes various matrices, but we focus on the two matrices of parameters of the $3$rd attention head of the $21$st layer, denoted with $W_Q(p) \in \R^{768 \times 64}$ and $W_K(p) \in \R^{768 \times 64}$. Let $X \in \R^{5000 \times 768}$ the embedding of the document we have selected at the beginning, our goal is to compute the leading singular value and vector of
\[ softmax \left( \frac{XW_Q(XW_K)^T}{\sqrt{64}}\right).\]
This is an important problem as, in many architectures, the attention matrix has very low-rank and a large fraction of the information is contained in the leading singular vectors (see for example \cite{Don2021}).
To simplify the  process, we save $XW_Q$ and $XW_K$ directly.

We extract the matrices $XW_Q$ and $XW_K$ at each step of the final training epoch (fourth epoch). This yields a discrete parameter set of $563$ sequential points, such that $p \in \{1, \dots, 563\}$. In \Cref{fig:trans_k40_freq20} and \Cref{tab:trans_k40_freq20} we report the results for $50$ snapshot points equal to $1,\dots,50$ and we do not update the snapshot matrix.  We see that the leading singular value does not change much during training. This may be connected to the fact that the network is not changing its attention core structure, even though the architecture is fine-tuning. This highlights that the fundamental structure learned by the model in the pre-training is still valid for our dataset. This phenomenon is observed in \cite[\S 4.3]{Kov2019} and \cite[\S 5]{Agh2021}.

In the plots, we note that the error is almost the same for all the parameters except the first ones. Furthermore, we observe that the \texttt{LU} variant of \texttt{SubApSnapEig} does not work for this example and overestimate the leading singular value by a factor of $2$. The accuracy for the other methods is of order $10^{-2}$ for the singular value and $10^{-1}$ for the singular vector, which is usually enough for machine learning tasks.

As our machine has enough RAM, we load all the matrices into RAM to avoid taking into account the time needed to load the matrices in memory. This shows that \texttt{SubApSnapEig} is around $20\times$ faster than \texttt{MOR} method for this example. This is also possible as we can first extract the rows from $XW_Q$, compute the rows of $XW_Q(XW_K)^T$ and then apply the softmax, as the softmax function applies row-wise.

% ==================== GRAPH FIG 1 ====================
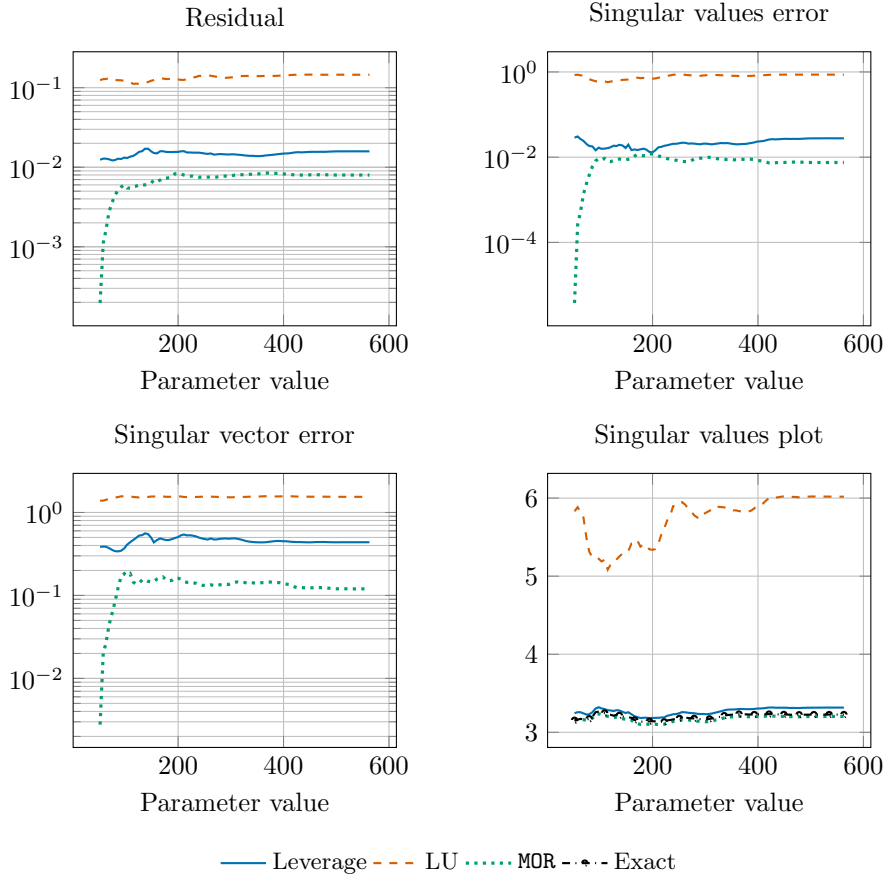
\begin{figure}[ht]
\centering
\begin{tikzpicture}
\begin{groupplot}[
    group style={group size=2 by 2, horizontal sep=2cm, vertical sep=0.15\textwidth},
    width=0.45\textwidth,
    height=0.4\textwidth,
    grid=both,
    legend style={
    at={(-0.3,-0.5)},
    anchor=south,
    legend columns=4,
    font=\small
    },
]

% Residual
\nextgroupplot[title={Residual}, xlabel={Parameter value}, ymode=log]
\addplot+ [mark=none] table [x=p, y=Res_Leverage] {experiments_new/transformer/k50_s0_tolInf_freqInf/res.dat};
\addplot+ [mark=none] table [x=p, y=Res_LU] {experiments_new/transformer/k50_s0_tolInf_freqInf/res.dat};
\addplot+ [mark=none] table [x=p, y=Res_POD] {experiments_new/transformer/k50_s0_tolInf_freqInf/res.dat};

% Eigenvalues
\nextgroupplot[title={Singular values error}, xlabel={Parameter value}, ymode=log]
\addplot+ [mark=none] table [x=p, y=ErrLam_Leverage] {experiments_new/transformer/k50_s0_tolInf_freqInf/eigval.dat};
\addplot+ [mark=none] table [x=p, y=ErrLam_LU] {experiments_new/transformer/k50_s0_tolInf_freqInf/eigval.dat};
\addplot+ [mark=none] table [x=p, y=ErrLam_POD] {experiments_new/transformer/k50_s0_tolInf_freqInf/eigval.dat};

% Eigenvectors
\nextgroupplot[title={Singular vector error}, xlabel={Parameter value}, ymode=log]
\addplot+ [mark=none] table [x=p, y=ErrAng_Leverage] {experiments_new/transformer/k50_s0_tolInf_freqInf/eigvec.dat};
\addplot+ [mark=none] table [x=p, y=ErrAng_LU] {experiments_new/transformer/k50_s0_tolInf_freqInf/eigvec.dat};
\addplot+ [mark=none] table [x=p, y=ErrAng_POD] {experiments_new/transformer/k50_s0_tolInf_freqInf/eigvec.dat};

% Plot
\nextgroupplot[title={Singular values plot}, xlabel={Parameter value}]
\addplot+ [mark=none] table [x=p, y=Lev_eig_1] {experiments_new/transformer/k50_s0_tolInf_freqInf/plot.dat}; \addlegendentry{Leverage}
\addplot+ [mark=none] table [x=p, y=LU_eig_1] {experiments_new/transformer/k50_s0_tolInf_freqInf/plot.dat}; \addlegendentry{LU}
\addplot+ [mark=none] table [x=p, y=POD_eig_1] {experiments_new/transformer/k50_s0_tolInf_freqInf/plot.dat}; \addlegendentry{\texttt{MOR}}
\addplot+ [mark = o, mark repeat=5] table [x=p, y=Ex_eig_1] {experiments_new/transformer/k50_s0_tolInf_freqInf/plot.dat}; \addlegendentry{Exact}

\end{groupplot}
\end{tikzpicture}
\caption{Results for experiments Transformer in \Cref{sub:trans} for $50$ snapshot points at the beginning.}\label{fig:trans_k40_freq20}
\end{figure}

\begin{table}[ht]
\centering
\caption{Computational times for experiments Transformer in \Cref{sub:eigcen} for $50$ snapshot points at the beginning.}\label{tab:trans_k40_freq20}
\begin{filecontents*}{experiments_new/transformer/k50_s0_tolInf_freqInf/times.dat}
Metodo	Offline_sec	Online_sec	Totale_sec	Ratio_Off_vs_POD	Ratio_On_vs_POD	Ratio_vs_POD
LU	13.311349	2.893466	16.204815	0.83144698526077	0.0315728480311515	0.150526850455686
Leverage	13.040347	4.711133	17.75148	0.814519790586538	0.0514068201470287	0.164893852557225
POD	16.009859	91.644124	107.653983	1	1	1
\end{filecontents*}
\pgfplotstabletypeset[times table style]{experiments_new/transformer/k50_s0_tolInf_freqInf/times.dat}
\end{table}

\section*{Acknowledgments}
During the preparation of this work, the authors made use of Gemini Pro 3.1 (Google) and limited use of Claude Sonnet 5 (Anthropic), and ChatGPT Edu (OpenAI) to assist in the following tasks: conducting bibliographic research (including the identification of application examples and locating \cite{Saf2010}), brainstorming and suggesting experimental designs, and writing, debugging, optimizing and commenting the source code used for the experiments (including the fine-tuning of the model in \Cref{sub:trans}).

\bibliographystyle{siamplain}
\bibliography{references}
\end{document}

%% file: ex_shared.tex
\usepackage{lipsum}
\usepackage{amsfonts}
\usepackage{graphicx}
\usepackage{epstopdf}
\usepackage{algpseudocode}
\usepackage{algorithm}
\ifpdf
  \DeclareGraphicsExtensions{.eps,.pdf,.png,.jpg}
\else
  \DeclareGraphicsExtensions{.eps}
\fi

\newcommand*\samethanks[1][\value{footnote}]{\footnotemark[#1]}

\newsiamremark{remark}{Remark}
\newsiamremark{hypothesis}{Hypothesis}
\crefname{hypothesis}{Hypothesis}{Hypotheses}
\newsiamthm{claim}{Claim}

\headers{Sequence of eigenvalue problems with snapshot and subsampling}{F. Hrobat and Y. Nakatsukasa}

\title{Solving a sequence of eigenvalue problems using snapshot and row subsampling\thanks{Version of \today.
\funding{FH is supported by an Mathematical Institute Scholarship of the Oxford Mathematical Institute and by a College Scholarship in Sciences of St Catherine's College, Oxford.
YN is supported by EPSRC grants EP/Y010086/1 and EP/Y030990/1. }}}

\author{Francesco Hrobat\thanks{Mathematical Institute, Oxford, United Kingdom
  (\email{francesco.hrobat@maths.ox.ac.uk}, 
  \email{yuji.nakatsukasa@maths.ox.ac.uk}).}
\and Yuji Nakatsukasa\samethanks
}

\usepackage{amsopn}

\makeatletter
\newcommand*{\addFileDependency}[1]{% argument=file name and extension
  \typeout{(#1)}% latexmk will find this if $recorder=0 (however, in that case, it will ignore #1 if it is a .aux or .pdf file etc and it exists! if it doesn't exist, it will appear in the list of dependents regardless)
  \@addtofilelist{#1}% if you want it to appear in \listfiles, not really necessary and latexmk doesn't use this
  \IfFileExists{#1}{}{\typeout{No file #1.}}% latexmk will find this message if #1 doesn't exist (yet)
}
\makeatother

%% file: references.bib
@misc{Jon2025,
      title={{SubApSnap: Solving parameter-dependent linear systems with a snapshot and subsampling}}, 
      author={Eleanor Jones and Yuji Nakatsukasa},
      year={2025},
  eprint        = {2510.04825},
  archivePrefix = {arXiv},
}

@article {Nak2024,
    AUTHOR = {Nakatsukasa, Yuji and Tropp, Joel A.},
     TITLE = {Fast and accurate randomized algorithms for linear systems and
              eigenvalue problems},
   JOURNAL = {SIAM J. Matrix Anal. Appl.},
  FJOURNAL = {SIAM Journal on Matrix Analysis and Applications},
    VOLUME = {45},
      YEAR = {2024},
    NUMBER = {2},
     PAGES = {1183--1214},
   MRCLASS = {65F10 (65F15)},
  MRNUMBER = {4760964},
MRREVIEWER = {Gang\ Wu},
}

@book {Saa2011,
    AUTHOR = {Saad, Yousef},
     TITLE = {Numerical methods for large eigenvalue problems},
    SERIES = {Classics in Applied Mathematics},
    VOLUME = {66},
   EDITION = {Revised},
 PUBLISHER = {Society for Industrial and Applied Mathematics (SIAM),
              Philadelphia, PA},
      YEAR = {2011},
     PAGES = {xvi+276},
      ISBN = {978-1-611970-72-2},
   MRCLASS = {65F15 (01A75 65-02 65F50)},
  MRNUMBER = {3396212},
}

@book{Sch2001,
  author    = {Schmid, Peter J. and Henningson, Dan S.},
  title     = {Stability and {T}ransition in {S}hear {F}lows},
  series    = {Applied Mathematical Sciences},
  volume    = {142},
  publisher = {Springer},
  address   = {New York},
  year      = {2001},
}

@article {LeB2005,
    AUTHOR = {Le Bris, Claude},
     TITLE = {Computational chemistry from the perspective of numerical
              analysis},
   JOURNAL = {Acta Numer.},
  FJOURNAL = {Acta Numerica},
    VOLUME = {14},
      YEAR = {2005},
     PAGES = {363--444},
      ISSN = {0962-4929,1474-0508},
   MRCLASS = {81V55 (92E20)},
  MRNUMBER = {2168346},
MRREVIEWER = {Brian\ L.\ Burrows},
}

@book {Lan2012,
    AUTHOR = {Langville, Amy N. and Meyer, Carl D.},
     TITLE = {Google's {P}age{R}ank and beyond: the science of search engine
              rankings},
   EDITION = {Paperback},
 PUBLISHER = {Princeton University Press, Princeton, NJ},
      YEAR = {2012},
     PAGES = {x+224},
}

@book{Has2009,
	author = {Trevor Hastie and Robert Tibshirani and Jerome Friedman},
	title = {The {E}lements of {S}tatistical {L}earning: {D}ata {M}ining, {I}nference and {P}rediction},
	publisher = {{S}pringer},
	edition = "2nd",
	_address = springer-ad,
	year = 2009
}

@book{Sch2002,

  author = {Sch\"{o}lkopf, Bernhard and Smola, Alexander J.},
  publisher = {MIT Press},
  series = {Adaptive computation and machine learning},
  title = {Learning with kernels : support vector machines, regularization, optimization, and beyond},
  year = 2002
}

@book {Chu1997,
    AUTHOR = {Chung, Fan R. K.},
     TITLE = {Spectral graph theory},
    SERIES = {CBMS Regional Conference Series in Mathematics},
    VOLUME = {92},
 PUBLISHER = {Conference Board of the Mathematical Sciences, Washington, DC;
              by the American Mathematical Society, Providence, RI},
      YEAR = {1997},
     PAGES = {xii+207},
      ISBN = {0-8218-0315-8},
   MRCLASS = {58G99 (05C50 35P05 46N20 47N20)},
  MRNUMBER = {1421568},
MRREVIEWER = {Robert\ Brooks},
}

@inproceedings{Bel2001,
author = {Belkin, Mikhail and Niyogi, Partha},
title = {Laplacian eigenmaps and spectral techniques for embedding and clustering},
year = {2001},
publisher = {MIT Press},
address = {Cambridge, MA, USA},
booktitle = {Proceedings of the 15th International Conference on Neural Information Processing Systems: Natural and Synthetic},
pages = {585–591},
numpages = {7},
location = {Vancouver, British Columbia, Canada},
series = {NIPS'01}
}

@ARTICLE{Sag2016,
       author = {{Sagun}, Levent and {Bottou}, Leon and {LeCun}, Yann},
        title = "{Eigenvalues of the Hessian in Deep Learning: Singularity and Beyond}",
      journal = {arXiv e-prints},
         year = 2016,
        pages = {arXiv:1611.07476},
archivePrefix = {arXiv},
       eprint = {1611.07476},
}

@inproceedings{Mit2013,
author = {Mitliagkas, Ioannis and Caramanis, Constantine and Jain, Prateek},
title = {{Memory limited, streaming PCA}},
year = {2013},
publisher = {Curran Associates Inc.},
address = {Red Hook, NY, USA},
booktitle = {Proceedings of the 27th International Conference on Neural Information Processing Systems - Volume 2},
pages = {2886–2894},
numpages = {9},
location = {Lake Tahoe, Nevada},
series = {NIPS'13}
}

@inproceedings{
Ana2019,
title={{Neural Operator: Graph Kernel Network for Partial Differential Equations}},
author={Anima Anandkumar and Kamyar Azizzadenesheli and Kaushik Bhattacharya and Nikola Kovachki and Zongyi Li and Burigede Liu and Andrew Stuart},
booktitle={ICLR 2020 Workshop on Integration of Deep Neural Models and Differential Equations},
year={2020},
}

@article {Buc2013,
    AUTHOR = {Buchan, A. G. and Pain, C. C. and Fang, F. and Navon, I. M.},
     TITLE = {A {POD} reduced-order model for eigenvalue problems with
              application to reactor physics},
   JOURNAL = {Internat. J. Numer. Methods Engrg.},
  FJOURNAL = {International Journal for Numerical Methods in Engineering},
    VOLUME = {95},
      YEAR = {2013},
    NUMBER = {12},
     PAGES = {1011--1032},
      ISSN = {0029-5981,1097-0207},
   MRCLASS = {65N25 (82D75)},
  MRNUMBER = {3100686},
}

@article {Alg2024,
    AUTHOR = {Alghamdi, Moataz and Bertrand, Fleurianne and Boffi, Daniele
              and Halim, Abdul},
     TITLE = {A data-driven method for parametric {PDE} eigenvalue problems
              using {G}aussian process with different covariance functions},
   JOURNAL = {Comput. Methods Appl. Math.},
  FJOURNAL = {Computational Methods in Applied Mathematics},
    VOLUME = {24},
      YEAR = {2024},
    NUMBER = {3},
     PAGES = {533--555},
      ISSN = {1609-4840,1609-9389},
   MRCLASS = {65N25 (35P15 62G08 78M34)},
  MRNUMBER = {4766709},
}

@article {Alg2025,
    AUTHOR = {Alghamdi, Moataz and Boffi, Daniele and Bonizzoni, Francesca},
     TITLE = {A greedy {MOR} method for the tracking of eigensolutions to
              parametric elliptic {PDE}s},
   JOURNAL = {J. Comput. Appl. Math.},
  FJOURNAL = {Journal of Computational and Applied Mathematics},
    VOLUME = {457},
      YEAR = {2025},
     PAGES = {Paper No. 116270, 16},
      ISSN = {0377-0427,1879-1778},
   MRCLASS = {65N25 (35B30 35P15 62K25 65N30 78M34)},
  MRNUMBER = {4801488},
MRREVIEWER = {Alexandre\ L.\ Madureira},
}

@article {Ber2023,
    AUTHOR = {Bertrand, Fleurianne and Boffi, Daniele and Halim, Abdul},
     TITLE = {A reduced order model for the finite element approximation of
              eigenvalue problems},
   JOURNAL = {Comput. Methods Appl. Mech. Engrg.},
  FJOURNAL = {Computer Methods in Applied Mechanics and Engineering},
    VOLUME = {404},
      YEAR = {2023},
     PAGES = {Paper No. 115696, 26},
      ISSN = {0045-7825,1879-2138},
   MRCLASS = {65N25 (65N30)},
  MRNUMBER = {4515207},
}

@article {Ber2023/2,
    AUTHOR = {Bertrand, Fleurianne and Boffi, Daniele and Halim, Abdul},
     TITLE = {Data-driven reduced order modeling for parametric {PDE}
              eigenvalue problems using {G}aussian process regression},
   JOURNAL = {J. Comput. Phys.},
  FJOURNAL = {Journal of Computational Physics},
    VOLUME = {495},
      YEAR = {2023},
     PAGES = {Paper No. 112503, 28},
      ISSN = {0021-9991,1090-2716},
   MRCLASS = {65J15 (65F55 65M70)},
  MRNUMBER = {4653319},
}

@article {Fum2016,
    AUTHOR = {Fumagalli, Ivan and Manzoni, Andrea and Parolini, Nicola and
              Verani, Marco},
     TITLE = {Reduced basis approximation and {\it a posteriori} error
              estimates for parametrized elliptic eigenvalue problems},
   JOURNAL = {ESAIM Math. Model. Numer. Anal.},
  FJOURNAL = {ESAIM. Mathematical Modelling and Numerical Analysis},
    VOLUME = {50},
      YEAR = {2016},
    NUMBER = {6},
     PAGES = {1857--1885},
      ISSN = {2822-7840,2804-7214},
   MRCLASS = {65N25 (65N15 65N30)},
  MRNUMBER = {3580125},
MRREVIEWER = {Srinivasan\ Kesavan},
}

@article {Hor2017,
    AUTHOR = {Horger, Thomas and Wohlmuth, Barbara and Dickopf, Thomas},
     TITLE = {Simultaneous reduced basis approximation of parameterized
              elliptic eigenvalue problems},
   JOURNAL = {ESAIM Math. Model. Numer. Anal.},
  FJOURNAL = {ESAIM. Mathematical Modelling and Numerical Analysis},
    VOLUME = {51},
      YEAR = {2017},
    NUMBER = {2},
     PAGES = {443--465},
      ISSN = {2822-7840,2804-7214},
   MRCLASS = {65N25 (35J25 65N30)},
  MRNUMBER = {3626406},
}

@book{Pin1985,
  author    = {Pinkus, Allan},
  title     = {n-{W}idths in {A}pproximation {T}heory},
  series    = {Ergebnisse der Mathematik und ihrer Grenzgebiete (3. Folge)},
  volume    = {7},
  publisher = {Springer-Verlag},
  address   = {Berlin, Heidelberg, New York},
  year      = {1985},
}

@article {Gui2024,
    AUTHOR = {Guido, Margherita and Kressner, Daniel and Ricci, Paolo},
     TITLE = {Subspace acceleration for a sequence of linear systems and
              application to plasma simulation},
   JOURNAL = {J. Sci. Comput.},
  FJOURNAL = {Journal of Scientific Computing},
    VOLUME = {99},
      YEAR = {2024},
    NUMBER = {3},
     PAGES = {Paper No. 68, 18},
      ISSN = {0885-7474,1573-7691},
   MRCLASS = {65M70 (76X05 82D10)},
  MRNUMBER = {4736333},
}

@article {Kre2011,
    AUTHOR = {Kressner, Daniel and Tobler, Christine},
     TITLE = {Low-rank tensor {K}rylov subspace methods for parametrized
              linear systems},
   JOURNAL = {SIAM J. Matrix Anal. Appl.},
  FJOURNAL = {SIAM Journal on Matrix Analysis and Applications},
    VOLUME = {32},
      YEAR = {2011},
    NUMBER = {4},
     PAGES = {1288--1316},
      ISSN = {0895-4798,1095-7162},
   MRCLASS = {65F10 (15A69)},
  MRNUMBER = {2854614},
MRREVIEWER = {Christian\ Wieners},
}

@book{Kat1976,
  author = "Kato, Tosio",
  title = "{{P}erturbation {T}heory for {L}inear {O}perators}",
  publisher = "Springer",
  address = "Berlin",
  year = "1976",
  edition = "2nd"
}

@book{For1981,
  title={Lectures on {R}iemann {S}urfaces},
  author={Forster, Otto},
  year={1981},
  publisher={Springer-Verlag},
  series={Graduate Texts in Mathematics},
  volume={81},
  address={New York},
}

@ARTICLE{Ber1993,
       author = {{Berkooz}, Gal and {Holmes}, Philip and {Lumley}, John L.},
        title = "{The proper orthogonal decomposition in the analysis of turbulent flows}",
      journal = {Annual Review of Fluid Mechanics},
         year = 1993,
        month = jan,
       volume = {25},
        pages = {539-575},
}

@article{Cha2000,
author = {Anindya Chatterjee},
 journal = {Current Science},
 number = {7},
 pages = {808--817},
 publisher = {Temporary Publisher},
 title = {An introduction to the proper orthogonal decomposition},
 volume = {78},
 year = {2000}
}

@article {Don2023,
    AUTHOR = {Dong, Yijun and Martinsson, Per-Gunnar},
     TITLE = {Simpler is better: a comparative study of randomized pivoting
              algorithms for {CUR} and interpolative decompositions},
   JOURNAL = {Adv. Comput. Math.},
  FJOURNAL = {Advances in Computational Mathematics},
    VOLUME = {49},
      YEAR = {2023},
    NUMBER = {4},
     PAGES = {Paper No. 66, 29},
      ISSN = {1019-7168,1572-9044},
   MRCLASS = {65F55},
  MRNUMBER = {4626444},
MRREVIEWER = {Miroslav\ T\ocirc uma},
}

@article {Osi2025,
    AUTHOR = {Osinsky, A. I.},
     TITLE = {Close to optimal column approximation using a single {SVD}},
   JOURNAL = {Linear Algebra Appl.},
  FJOURNAL = {Linear Algebra and its Applications},
    VOLUME = {725},
      YEAR = {2025},
     PAGES = {359--377},
      ISSN = {0024-3795,1873-1856},
   MRCLASS = {65F55 (15A23 68W25)},
  MRNUMBER = {4935617},
MRREVIEWER = {Pablo\ Sot}
}

@article {Dri2012,
    AUTHOR = {Drineas, Petros and Magdon-Ismail, Malik and Mahoney, Michael
              W. and Woodruff, David P.},
     TITLE = {Fast approximation of matrix coherence and statistical
              leverage},
   JOURNAL = {J. Mach. Learn. Res.},
  FJOURNAL = {Journal of Machine Learning Research (JMLR)},
    VOLUME = {13},
      YEAR = {2012},
     PAGES = {3475--3506},
      ISSN = {1532-4435,1533-7928},
   MRCLASS = {65F15 (60C05 60E15 68W20)},
  MRNUMBER = {3033372},
MRREVIEWER = {Raffaella\ Pavani},
}

@article {Der2021,
    AUTHOR = {Derezi\'nski, Micha\l{} and Mahoney, Michael W.},
     TITLE = {Determinantal point processes in randomized numerical linear
              algebra},
   JOURNAL = {Notices Amer. Math. Soc.},
  FJOURNAL = {Notices of the American Mathematical Society},
    VOLUME = {68},
      YEAR = {2021},
    NUMBER = {1},
     PAGES = {34--45},
      ISSN = {0002-9920,1088-9477},
   MRCLASS = {65F40 (60G55)},
  MRNUMBER = {4202314},
}

@article {Cor2025,
    AUTHOR = {Cortinovis, Alice and Kressner, Daniel},
     TITLE = {Adaptive randomized pivoting for column subset selection,
              {DEIM}, and low-rank approximation},
   JOURNAL = {SIAM J. Matrix Anal. Appl.},
  FJOURNAL = {SIAM Journal on Matrix Analysis and Applications},
    VOLUME = {47},
      YEAR = {2026},
    NUMBER = {1},
     PAGES = {25--47},
   MRCLASS = {65F55 (65C20)},
  MRNUMBER = {5009256},
}

@article {Kil2006,
    AUTHOR = {Kilmer, Misha E. and de Sturler, Eric},
     TITLE = {Recycling subspace information for diffuse optical tomography},
   JOURNAL = {SIAM J. Sci. Comput.},
  FJOURNAL = {SIAM Journal on Scientific Computing},
    VOLUME = {27},
      YEAR = {2006},
    NUMBER = {6},
     PAGES = {2140--2166},
      ISSN = {1064-8275,1095-7197},
   MRCLASS = {65F10 (65N22)},
  MRNUMBER = {2211444},
MRREVIEWER = {Jan\ Z\'itko},
}

@article {Par2006,
    AUTHOR = {Parks, Michael L. and de Sturler, Eric and Mackey, Greg and
              Johnson, Duane D. and Maiti, Spandan},
     TITLE = {Recycling {K}rylov subspaces for sequences of linear systems},
   JOURNAL = {SIAM J. Sci. Comput.},
  FJOURNAL = {SIAM Journal on Scientific Computing},
    VOLUME = {28},
      YEAR = {2006},
    NUMBER = {5},
     PAGES = {1651--1674},
      ISSN = {1064-8275,1095-7197},
   MRCLASS = {65F10},
  MRNUMBER = {2272183},
MRREVIEWER = {Petr\ Tich\'y},
}

@article {Sim2010,
    AUTHOR = {Simoncini, V.},
     TITLE = {Extended {K}rylov subspace for parameter dependent systems},
   JOURNAL = {Appl. Numer. Math.},
  FJOURNAL = {Applied Numerical Mathematics. An IMACS Journal},
    VOLUME = {60},
      YEAR = {2010},
    NUMBER = {5},
     PAGES = {550--560},
      ISSN = {0168-9274,1873-5460},
   MRCLASS = {65F60 (15A24 65F10 65F30)},
  MRNUMBER = {2610803},
}

@article {Cor2025/2,
    AUTHOR = {Correnty, Siobh\'an and Freitag, Melina A. and Soodhalter,
              Kirk M.},
     TITLE = {Chebyshev {HOPGD} with sparse grid sampling for parameterized
              linear systems},
   JOURNAL = {Calcolo},
  FJOURNAL = {Calcolo. A Quarterly on Numerical Analysis and Theory of
              Computation},
    VOLUME = {62},
      YEAR = {2025},
    NUMBER = {3},
     PAGES = {Paper No. 28, 29},
      ISSN = {0008-0624,1126-5434},
   MRCLASS = {65F10 (65F55 65N22)},
  MRNUMBER = {4933926},
}

@article {Cor2023,
    AUTHOR = {Correnty, Siobh\'an and Jarlebring, Elias and Szyld, Daniel
              B.},
     TITLE = {Preconditioned {C}hebyshev {B}i{CG} method for parameterized
              linear systems},
   JOURNAL = {Electron. Trans. Numer. Anal.},
  FJOURNAL = {Electronic Transactions on Numerical Analysis},
    VOLUME = {58},
      YEAR = {2023},
     PAGES = {629--656},
      ISSN = {1068-9613},
   MRCLASS = {65F10 (15A06 65F08 65F50 65N22 65P99)},
  MRNUMBER = {4681081},
}

@article {Jar2022,
    AUTHOR = {Jarlebring, Elias and Correnty, Siobh\'an},
     TITLE = {Infinite {GMRES} for parameterized linear systems},
   JOURNAL = {SIAM J. Matrix Anal. Appl.},
  FJOURNAL = {SIAM Journal on Matrix Analysis and Applications},
    VOLUME = {43},
      YEAR = {2022},
    NUMBER = {3},
     PAGES = {1382--1405},
      ISSN = {0895-4798,1095-7162},
   MRCLASS = {65F10 (35G15 65F50)},
  MRNUMBER = {4468619},
MRREVIEWER = {Zheng\ Sheng\ Wang},
}

@article {Mar2006,
    AUTHOR = {Markovinovi\'c, R. and Jansen, J. D.},
     TITLE = {Accelerating iterative solution methods using reduced-order
              models as solution predictors},
   JOURNAL = {Internat. J. Numer. Methods Engrg.},
  FJOURNAL = {International Journal for Numerical Methods in Engineering},
    VOLUME = {68},
      YEAR = {2006},
    NUMBER = {5},
     PAGES = {525--541},
      ISSN = {0029-5981,1097-0207},
   MRCLASS = {65M99 (76M25 76S05)},
  MRNUMBER = {2261401},
}

@article {Cha2010,
    AUTHOR = {Chaturantabut, Saifon and Sorensen, Danny C.},
     TITLE = {Nonlinear model reduction via discrete empirical
              interpolation},
   JOURNAL = {SIAM J. Sci. Comput.},
  FJOURNAL = {SIAM Journal on Scientific Computing},
    VOLUME = {32},
      YEAR = {2010},
    NUMBER = {5},
     PAGES = {2737--2764},
      ISSN = {1064-8275,1095-7197},
   MRCLASS = {65L99 (35J60 35K55 65F99 65M99)},
  MRNUMBER = {2684735},
MRREVIEWER = {Sebastian\ Reich},
}

@book {Ste2001,
    AUTHOR = {Stewart, G. W.},
     TITLE = {Matrix algorithms. {V}ol. {II}},
      NOTE = {Eigensystems},
 PUBLISHER = {Society for Industrial and Applied Mathematics (SIAM),
              Philadelphia, PA},
      YEAR = {2001},
     PAGES = {xx+469},
      ISBN = {0-89871-503-2},
   MRCLASS = {65-01 (65F15)},
  MRNUMBER = {1853468},
}

@article {Mar2020,
    AUTHOR = {Martinsson, Per-Gunnar and Tropp, Joel A.},
     TITLE = {Randomized numerical linear algebra: foundations and
              algorithms},
   JOURNAL = {Acta Numer.},
  FJOURNAL = {Acta Numerica},
    VOLUME = {29},
      YEAR = {2020},
     PAGES = {403--572},
      ISSN = {0962-4929,1474-0508},
   MRCLASS = {65F99 (60E15 62-08)},
  MRNUMBER = {4189294},
MRREVIEWER = {Bruno\ Carpentieri},
}

@article {Soo2020,
    AUTHOR = {Soodhalter, Kirk M. and de Sturler, Eric and Kilmer, Misha E.},
     TITLE = {A survey of subspace recycling iterative methods},
   JOURNAL = {GAMM-Mitt.},
  FJOURNAL = {GAMM-Mitteilungen},
    VOLUME = {43},
      YEAR = {2020},
    NUMBER = {4},
     PAGES = {e202000016, 29},
      ISSN = {0936-7195,1522-2608},
   MRCLASS = {65F10},
  MRNUMBER = {4175151},
    }

@article {Kun1999,
    AUTHOR = {Kunisch, K. and Volkwein, S.},
     TITLE = {Control of the {B}urgers equation by a reduced-order approach
              using proper orthogonal decomposition},
   JOURNAL = {J. Optim. Theory Appl.},
  FJOURNAL = {Journal of Optimization Theory and Applications},
    VOLUME = {102},
      YEAR = {1999},
    NUMBER = {2},
     PAGES = {345--371},
   MRCLASS = {49M27 (49J20 93C20)},
  MRNUMBER = {1706822},
MRREVIEWER = {Walter\ Alt},
}

@article {Str1999,
    AUTHOR = {Strang, Gilbert},
     TITLE = {The discrete cosine transform},
   JOURNAL = {SIAM Rev.},
  FJOURNAL = {SIAM Review},
    VOLUME = {41},
      YEAR = {1999},
    NUMBER = {1},
     PAGES = {135--147},
      ISSN = {1095-7200,0036-1445},
   MRCLASS = {42C15 (65T20)},
  MRNUMBER = {1669796},
MRREVIEWER = {Krzysztof\ Nowak},
}

@article {Dem2019,
    AUTHOR = {Demanet, Laurent and Townsend, Alex},
     TITLE = {Stable extrapolation of analytic functions},
   JOURNAL = {Found. Comput. Math.},
  FJOURNAL = {Foundations of Computational Mathematics. The Journal of the
              Society for the Foundations of Computational Mathematics},
    VOLUME = {19},
      YEAR = {2019},
    NUMBER = {2},
     PAGES = {297--331},
      ISSN = {1615-3375,1615-3383},
   MRCLASS = {41A10 (30E10 65D05)},
  MRNUMBER = {3937956},
MRREVIEWER = {Yi\ Li},
}

@article {Saf2010,
    AUTHOR = {Saff, E. B.},
     TITLE = {Logarithmic potential theory with applications to
              approximation theory},
   JOURNAL = {Surv. Approx. Theory},
  FJOURNAL = {Surveys in Approximation Theory},
    VOLUME = {5},
      YEAR = {2010},
     PAGES = {165--200},
      ISSN = {1555-578X},
   MRCLASS = {31-02 (30E10 31A15 41A20)},
  MRNUMBER = {2734198},
MRREVIEWER = {Anders\ Gustafsson},
}

@book {Tre2020,
    AUTHOR = {Trefethen, Lloyd N.},
     TITLE = {Approximation theory and approximation practice},
   EDITION = {Extended},
 PUBLISHER = {Society for Industrial and Applied Mathematics (SIAM),
              Philadelphia, PA},
      YEAR = {[2020] \copyright 2020},
     PAGES = {xi+363},
      ISBN = {978-1-611975-93-2},
   MRCLASS = {41-01 (41-04 65-01)},
  MRNUMBER = {4050406},
MRREVIEWER = {Ralitza\ Kovacheva},
}

@article {Ban2022,
    AUTHOR = {Banerjee, Sayan and Olvera-Cravioto, Mariana},
     TITLE = {Page{R}ank asymptotics on directed preferential attachment
              networks},
   JOURNAL = {Ann. Appl. Probab.},
  FJOURNAL = {The Annals of Applied Probability},
    VOLUME = {32},
      YEAR = {2022},
    NUMBER = {4},
     PAGES = {3060--3084},
      ISSN = {1050-5164,2168-8737},
   MRCLASS = {05C80 (41A60 60B10 60J80)},
  MRNUMBER = {4474527},
MRREVIEWER = {Mei\ Yin},
}

@inproceedings{Mei2010,
  author    = {Mei, Qiaozhu and Guo, Jian and Radev, Dragomir},
  title     = {DivRank: the interplay of prestige and diversity in information networks},
  booktitle = {Proceedings of the 16th ACM SIGKDD International Conference on Knowledge Discovery and Data Mining},
  series    = {KDD '10},
  year      = {2010},
  location  = {Washington, DC, USA},
  pages     = {1009--1018},
  publisher = {Association for Computing Machinery},
  address   = {New York, NY, USA}
}

@article {Lev2006,
    AUTHOR = {Levenberg, Norm},
     TITLE = {Approximation in {$\mathbb{C}^N$}},
   JOURNAL = {Surv. Approx. Theory},
  FJOURNAL = {Surveys in Approximation Theory},
    VOLUME = {2},
      YEAR = {2006},
     PAGES = {92--140},
      ISSN = {1555-578X},
   MRCLASS = {32E30},
  MRNUMBER = {2276419},
MRREVIEWER = {S\l awomir\ Ko\l odziej},
}

@article {Hal2011,
    AUTHOR = {Halko, N. and Martinsson, P. G. and Tropp, J. A.},
     TITLE = {Finding structure with randomness: probabilistic algorithms
              for constructing approximate matrix decompositions},
   JOURNAL = {SIAM Rev.},
  FJOURNAL = {SIAM Review},
    VOLUME = {53},
      YEAR = {2011},
    NUMBER = {2},
     PAGES = {217--288},
      ISSN = {1095-7200,0036-1445},
   MRCLASS = {65F30 (60B20 68W20)},
  MRNUMBER = {2806637},
MRREVIEWER = {Thomas\ K.\ Huckle},
}

@book{Ras2005,
  title     = {{Gaussian Processes for Machine Learning}},
  author    = {Rasmussen, Carl Edward and Williams, Christopher K. I.},
  year      = {2005},
  publisher = {The MIT Press},
  series    = {Adaptive Computation and Machine Learning},
  address   = {Cambridge, MA, USA},
  isbn      = {9780262182539},
}

@article {Rea1983,
    AUTHOR = {Reade, J. B.},
     TITLE = {Eigenvalues of positive definite kernels},
   JOURNAL = {SIAM J. Math. Anal.},
  FJOURNAL = {SIAM Journal on Mathematical Analysis},
    VOLUME = {14},
      YEAR = {1983},
    NUMBER = {1},
     PAGES = {152--157},
      ISSN = {0036-1410},
   MRCLASS = {47A70 (47B38)},
  MRNUMBER = {686242},
MRREVIEWER = {Bernd\ Carl},
}

@article {Ha1986,
    AUTHOR = {Ha, Chung Wei},
     TITLE = {Eigenvalues of differentiable positive definite kernels},
   JOURNAL = {SIAM J. Math. Anal.},
  FJOURNAL = {SIAM Journal on Mathematical Analysis},
    VOLUME = {17},
      YEAR = {1986},
    NUMBER = {2},
     PAGES = {415--419},
      ISSN = {0036-1410},
   MRCLASS = {47A70 (47G05)},
  MRNUMBER = {826702},
MRREVIEWER = {M.\ Tvrd\'y},
}

@book{Al-2008,
  title={{Sturm--Liouville Theory and its Applications}},
  author={Al-Gwaiz, M.A.},
  year={2008},
  publisher={Springer London}
}

@article{Hen2026,
  author  = {Henrot, A. and Lemenant, A. and Privat, Y.},
  title   = {Minimization of the first eigenvalue for the {Lamé} system},
  journal = {Annali Scuola Normale Superiore - Classe di Scienze},
  year    = {2026},
}

@article{Kha2026,
  author   = {Khan, Arbaz and Lepe, Felipe and Mora, David and Vellojin, Jesus},
  title    = {Finite element analysis of the nearly incompressible linear elasticity eigenvalue problem with variable coefficients},
  journal  = {Numerische Mathematik},
  year     = {2026},
  volume   = {158},
  number   = {4},
  pages    = {1473--1512},
}

@ARTICLE{Lec1998,
  author={Lecun, Y. and Bottou, L. and Bengio, Y. and Haffner, P.},
  journal={Proceedings of the IEEE}, 
  title={Gradient-based learning applied to document recognition}, 
  year={1998},
  volume={86},
  number={11},
  pages={2278-2324},
  }

@book{New2010,
  title={{Networks: An Introduction}},
  author={Newman, Mark E. J.},
  year={2010},
  publisher={Oxford University Press},
  address={Oxford, UK}
}

@article{Dav2011,
  author = {Davis, Timothy A. and Hu, Yifan},
  title = {{The University of Florida Sparse Matrix Collection}},
  journal = {ACM Transactions on Mathematical Software},
  volume = {38},
  number = {1},
  year = {2011},
}

@book{Bad2013,
  editor = {Bader, David A. and Meyerhenke, Henning and Sanders, Peter and Wagner, Dorothea},
  title = {{Graph Partitioning and Graph Clustering: 10th DIMACS Implementation Challenge Workshop}},
  series = {Contemporary Mathematics},
  volume = {588},
  publisher = {American Mathematical Society},
  year = {2013}
}

@inproceedings{Vas2017,
author = {Vaswani, Ashish and Shazeer, Noam and Parmar, Niki and Uszkoreit, Jakob and Jones, Llion and Gomez, Aidan N. and Kaiser, {\L}ukasz and Polosukhin, Illia},
title = {Attention is all you need},
year = {2017},
booktitle = {Proceedings of the 31st International Conference on Neural Information Processing Systems},
pages = {6000–6010},
numpages = {11},
location = {Long Beach, California, USA},
series = {NIPS'17}
}

@misc{War2024,
      title={{Smarter, Better, Faster, Longer: A Modern Bidirectional Encoder for Fast, Memory Efficient, and Long Context Finetuning and Inference}}, 
      author={Benjamin Warner and Antoine Chaffin and Benjamin Clavié and Orion Weller and Oskar Hallström and Said Taghadouini and Alexis Gallagher and Raja Biswas and Faisal Ladhak and Tom Aarsen and Nathan Cooper and Griffin Adams and Jeremy Howard and Iacopo Poli},
      year={2024},
      eprint={2412.13663},
      archivePrefix={arXiv},
}

@inproceedings{Cha2022,
    title = "{{L}ex{GLUE}: A Benchmark Dataset for Legal Language Understanding in {E}nglish}",
    author = "Chalkidis, Ilias  and
      Jana, Abhik  and
      Hartung, Dirk  and
      Bommarito, Michael  and
      Androutsopoulos, Ion  and
      Katz, Daniel  and
      Aletras, Nikolaos",
    booktitle = "Proceedings of the 60th Annual Meeting of the Association for Computational Linguistics (Volume 1: Long Papers)",
    month = may,
    year = "2022",
    address = "Dublin, Ireland",
    publisher = "Association for Computational Linguistics",
    pages = "4310--4330"
}

@inproceedings{Kov2019,
    title = "{Revealing the Dark Secrets of {BERT}}",
    author = "Kovaleva, Olga  and
      Romanov, Alexey  and
      Rogers, Anna  and
      Rumshisky, Anna",
    booktitle = "Proceedings of the 2019 Conference on Empirical Methods in Natural Language Processing and the 9th International Joint Conference on Natural Language Processing (EMNLP-IJCNLP)",
    year = "2019",
    publisher = "Association for Computational Linguistics",
    pages = "4365--4374"
}

@inproceedings{Agh2021,
    title = "{Intrinsic Dimensionality Explains the Effectiveness of Language Model Fine-Tuning}",
    author = "Aghajanyan, Armen  and
      Zettlemoyer, Luke  and
      Gupta, Sonal",
    booktitle = "Proceedings of the 59th Annual Meeting of the Association for Computational Linguistics",
    year = "2021",
    publisher = "Association for Computational Linguistics",
    pages = "7319--7328"
}

@inproceedings{Don2021,
  author    = {Dong, Yihe and Cordonnier, Jean-Baptiste and Loukas, Andreas},
  title     = {Attention is Not All You Need: Pure Attention Loses Rank Doubly Exponentially with Depth},
  booktitle = {Proceedings of the 38th International Conference on Machine Learning},
  series    = {Proceedings of Machine Learning Research},
  volume    = {139},
  pages     = {2793--2803},
  year      = {2021},
  publisher = {PMLR},
}

@article{Bri1998,
title = {The anatomy of a large-scale hypertextual {W}eb search engine},
journal = {Computer Networks and ISDN Systems},
volume = {30},
number = {1-7},
pages = {107-117},
year = {1998},
author = {Sergey Brin and Lawrence Page},
}

@article {Sir2016,
    AUTHOR = {Sirkovi\'c, Petar and Kressner, Daniel},
     TITLE = {Subspace acceleration for large-scale parameter-dependent
              {H}ermitian eigenproblems},
   JOURNAL = {SIAM J. Matrix Anal. Appl.},
  FJOURNAL = {SIAM Journal on Matrix Analysis and Applications},
    VOLUME = {37},
      YEAR = {2016},
    NUMBER = {2},
     PAGES = {695--718},
      ISSN = {0895-4798,1095-7162},
   MRCLASS = {65F15 (15B57 47A55)},
  MRNUMBER = {3504990},
MRREVIEWER = {Rafikul\ Alam},
}

@misc{patera2007reduced,
  title={Reduced basis approximation and a posteriori error estimation for parametrized partial differential equations},
  author={Patera, Anthony T and Rozza, Gianluigi},
  year={2006},
  publisher={MIT Cambridge}
}

@article{maday2002priori,
  title={A priori convergence theory for reduced-basis approximations of single-parameter elliptic partial differential equations},
  author={Maday, Yvon and Patera, Anthony T and Turinici, Gabriel},
  journal={Journal of Scientific Computing},
  volume={17},
  number={1},
  pages={437--446},
  year={2002},
  publisher={Springer}
}

@incollection{hesthaven2015parametrized,
  title={Parametrized differential equations},
  author={Hesthaven, Jan S and Rozza, Gianluigi and Stamm, Benjamin},
  booktitle={Certified Reduced Basis Methods for Parametrized Partial Differential Equations},
  pages={15--25},
  year={2015},
  publisher={Springer}
}

@book{mason2010chebyshev,
  title={Chebyshev Polynomials},
  author={Mason, John C. and Handscomb, David C.},
  year={2002},
  publisher={Chapman \& Hall}
}

@book {Par1998,
    AUTHOR = {Parlett, Beresford N.},
     TITLE = {The symmetric eigenvalue problem},
    SERIES = {Classics in Applied Mathematics},
    VOLUME = {20},
      NOTE = {Corrected reprint of the 1980 original},
 PUBLISHER = {Society for Industrial and Applied Mathematics (SIAM),
              Philadelphia, PA},
      YEAR = {1998},
     PAGES = {xxiv+398},
      ISBN = {0-89871-402-8},
   MRCLASS = {65F15 (15A18)},
  MRNUMBER = {1490034},
MRREVIEWER = {F.\ Szidarovszky},
}
